\nonstopmode
\documentclass[fms, times, draft]{cuparticle}

\usepackage{amsmath, amsthm, amssymb}
\usepackage[english]{babel}
\usepackage[utf8]{inputenc}
\usepackage{mathrsfs}                     
\usepackage{bbm}                          
\usepackage[all]{xy}                      
\usepackage{url}
\usepackage{enumerate}

\theoremstyle{plain}
\newtheorem{theorem}[subsection]{Theorem}
\newtheorem{lemma}[subsection]{Lemma}
\newtheorem{corollary}[subsection]{Corollary}

\newtheorem*{lemma*}{Lemma}

\theoremstyle{definition}

\newtheorem{remark}[subsection]{Remark}

\newtheorem*{remark*}{Remark}
\newtheorem*{example*}{Example}

\def\al{\alpha} 
\def\be{\beta}

\def\th{\theta} 
 
\def\ka{\kappa}

\def\si{\sigma} 
\def\ta{\tau} 
\def\ph{\varphi} 
 
\def\ps{\psi} 
 
\def\Ga{\Gamma}

\def\Ph{\Phi} 
\def\Ps{\Psi}

\def\o{\circ} 
\def\i{^{-1}} 
\def\x{\times}
\def\p{\partial}

\def\R{{\mathbb R}}

\def\exp{\operatorname{exp}}

\def\Imm{\operatorname{Imm}}

\let\on=\operatorname

\let\wh=\widehat

\let\mc=\mathcal

\newcommand{\ud}{\,\mathrm{d}}

\volume{1}
\doi{10.111}

\begin{document}

\authorheadline{M. Bruveris, P.W. Michor  and D. Mumford}
\runningtitle{Geodesic Completeness for Sobolev Metrics on $\on{Imm}(S^1,\R^2)$}

\begin{frontmatter}

\title{Geodesic Completeness for Sobolev Metrics on the Space of Immersed Plane Curves}

\author[1]{Martins Bruveris}
\cormark[1]
\cortext[1]{Corresponding author}
\ead{martins.bruveris@epfl.ch}
\address[1]{Institut de mathématiques, EPFL, Lausanne 1015, Switzerland}

\author[2]{Peter W. Michor}
\address[2]{Fakult\"at f\"ur Mathematik, Universit\"at Wien,
Oskar-Morgenstern-Platz 1, Wien 1090, Austria}

\author[3]{David Mumford}
\address[3]{Division of Applied Mathematics, Brown University,
Box F, Providence, RI 02912, USA}

\received{Some time}
\accepted{Some time later}

\begin{abstract}
We study properties of Sobolev-type metrics on the space of immersed plane curves. 
We show that the geodesic equation for Sobolev-type metrics with constant coefficients of order 2 
and higher is globally well-posed for smooth initial data as well as initial data in certain 
Sobolev spaces. Thus the space of closed plane curves equipped with such a metric is geodesically 
complete. We find lower bounds for the geodesic distance in terms of curvature and its derivatives.   
\end{abstract}

\end{frontmatter}

\section{Introduction}
Sobolev-type metrics on the space of plane immersed curves were independently introduced in 
\cite{Charpiat2007,Michor2006c, Mennucci2007}. They are used in computer vision, shape 
classification and tracking, mainly in the form of their induced metric on shape space, which is 
the orbit space under the action of the reparameterization group. 
See \cite{Kurtek2012,Sundaramoorthi2011} for applications of 
Sobolev-type metrics and \cite{Bauer2013_preprint,Michor2007} for an overview of their mathematical 
properties. Sobolev-type metrics were also generalized to immersions of higher dimensional 
manifolds in \cite{Bauer2011b,Bauer2012d}.     

It was shown in \cite{Michor2007} that the geodesic equation of a Sobolev-type metric of order 
$n\geq 1$ is locally well-posed and this result was extended in \cite{Bauer2011b} to a larger class 
of metrics and immersions of arbitrary dimension. The main result of this paper is to show global 
well-posedness of the geodesic equation for Sobolev-type metrics of order $n\geq 2$ with constant 
coefficients. In particular we prove the following theorem: 

\begin{theorem}\label{mainTheorem}
Let $n\geq 2$ and the metric $G$ on $\on{Imm}(S^1,\R^2)$ be given by
\[
G_c(h,k) = \int_{S^1} \sum_{j=0}^n a_j \langle D_s^j h, D_s^j k \rangle \ud s\,,
\]
with $a_j\geq 0$ and $a_0,a_n \neq 0$. Given initial conditions $(c_0, u_0) \in T\on{Imm}(S^1,\R^2)$ the solution of the geodesic equation 
\begin{equation*}
\begin{split}
\p_t \left(\sum_{j=0}^n (-1)^j \,|c'|\, D_s^{2j} c_t\right) &=
-\frac{a_0}2 \,|c'|\, D_s\left( \langle c_t, c_t \rangle v \right) \\
&\qquad{} + \sum_{k=1}^n \sum_{j=1}^{2k-1} (-1)^{k+j} \frac{a_k}{2}\, |c'|\, D_s
\left(\langle D_s^{2k-j} c_t, D_s^j c_t \rangle v \right)\,.
\end{split}
\end{equation*}
for the metric $G$ with initial values $(c_0,u_0)$ exists for all time.
\end{theorem}

Here $\Imm(S^1,\mathbb R^2)$ denotes the 
space of all smooth, closed, plane curves with nowhere zero tangent vectors; this space is open in 
$C^\infty(S^1,\mathbb R^2)$. We assume that $c\in \Imm(S^1,\mathbb R^2)$ and $h$, $k$ are vector 
fields along $c$, $\ud s=|c'|\ud\th$ is the arc-length measure, $D_s=\frac 1{|c'|}\p_\th$ is the 
derivative with respect to arc-length, $v = c'/|c'|$ is the unit length tangent vector to $c$
and $\langle \;,\; \rangle$ is the Euclidean inner product on $\mathbb R^2$.

Thus if $G$ is a Sobolev-type metric of order at least 2, then the Riemannian manifold 
$(\on{Imm}(S^1,\R^2), G)$ is geodesically complete. 
If the Sobolev-type metric is invariant under the reparameterization group $\on{Diff}(S^1)$, 
also the induced metric on shape space $\Imm(S^1,\mathbb R^2)/\on{Diff}(S^1)$ is geodesically complete.
The latter space is an infinite dimensional orbifold; see  
\cite[2.5 and 2.10]{Michor2006c}.

Theorem \ref{mainTheorem} seems to be the first result about geodesic completeness on manifolds 
of mappings outside the realm of diffeomorphism groups and manifolds of metrics.
In the first paragraph of \cite[p. 140]{Ebin1970} a proof is sketched that a right 
invariant $H^s$-metric on the group of volume preserving diffeomorphisms on a compact manifold $M$ is 
geodesically complete, if $s\ge \dim(M)/2+1$. In \cite{TrouveYounes05} there is an implicit 
result that a topological group of diffeomorphisms constructed from a reproducing kernel Hilbert 
space of vector 
fields whose reproducing kernel is at least $C^1$, is geodesically complete. For a certain metric 
on a group of diffeomorphisms on $\mathbb R^n$ with $C^1$ kernel geodesic completeness is shown in 
\cite[Thm.~2]{Michor2013}. Metric completeness and existence of minimizing geodesics have also been studied on the diffeomorphism group in \cite{Bruveris2014}. The manifold of all Riemannian metrics with fixed volume form is 
geodesically complete for the $L^2$-metric (also called the Ebin metric).

Sobolev-type metrics of order 1 are not geodesically complete, since it is possible to shrink a 
circle to a point along a geodesic in finite time, see \cite[Sect.~6.1]{Michor2007}. Similarly a 
Sobolev metric of order 2 or higher with both $a_0, a_1=0$ is a geodesically incomplete metric on 
the space $\on{Imm}(S^1,\R^2)/\on{Tra}$ of plane curves modulo translations. In this case it is 
possible to blow up a circle along a geodesic to infinity in finite time; see Rem.\  
\ref{rem:incomplete}.     

In order to prove long-time existence of geodesics, we need to study properties of the geodesic distance. 
In particular we show the following theorem regarding continuity of curvature $\ka$ and its derivatives.

\begin{theorem}
\label{thm:l2_bound_intro}
Let $G$ be a Sobolev-type metric of order $n\geq 2$ with constant coefficients and $\on{dist}^G$ 
the induced geodesic distance. If $0 \leq k \leq n-2$, then the functions 
\begin{align*}
D_s^k (\ka) \sqrt{|c'|} &: (\on{Imm}(S^1,\R^2), \on{dist}^G) \to L^2(S^1,\R) \\
D_s^{k+1} (\log|c'|) \sqrt{|c'|} &: (\on{Imm}(S^1,\R^2), \on{dist}^G) \to L^2(S^1,\R)
\end{align*}
are continuous and Lipschitz continuous on every metric ball.
\end{theorem}

A similar statement can be derived for the $L^\infty$-continuity of curvature and its derivatives; 
see Rem.\ \ref{rem:Linfty_continuity}.

The full proof of Thm.~\ref{mainTheorem} is surprisingly complicated. One reason is that we have 
to work on the Sobolev completion (always with respect to the original parameter $\th$ in $S^1$)
of the space of immersions in order to apply results on ODEs on 
Banach spaces. Here the operators (and their inverses and adjoints) acquire non-smooth coefficients. 
Since we we want the Sobolev order as low as possible, the geodesic equation involves $H^{-n}$; see Sect.~\ref{lem:Ds_smooth_sobolev}.
Eventually we use that the metric operator has constant coefficients.
We have to use estimates with precise constants which are uniformly bounded on metric balls.

In \cite{Bauer2011b} the authors studied Sobolev metrics on immersions of higher dimensional 
manifolds. One might hope that similar methods to those used in this article can be applied to show 
the geodesic completeness of the spaces $\on{Imm}(M,N)$ with $M$ compact and $(N, \bar{g})$ a 
suitable Riemannian manifold. A crucial ingredient in the proof for plane curves are the Sobolev 
inequalities Lem. \ref{lem:poincare} and Lem. \ref{lem:poincare2} with explicit constants, which 
only depend on the curve through the length. The lack of such inequalities for general $M$ will one 
of the factors complicating life in higher dimensions.      

\section{Background Material and Notation}

\subsection{The Space of Curves}

The space
\[
\on{Imm}(S^1, \R^2) = \left\{ c \in C^\infty(S^1, \R^2) \,:\, c'(\th) \neq 0 \right\}
\]
of immersions is an open set in the Fr\'echet space $C^\infty(S^1, \R^2)$ with respect to the $C^\infty$-topology and thus itself a smooth Fr\'echet manifold. The tangent space of $\on{Imm}(S^1, \R^2)$ at the point $c$ consists of all vector fields along the curve $c$. It can be described as the space of sections of the pullback bundle $c^\ast T\R^2$,
\begin{equation*}
T_c \on{Imm}(S^1,\R^2) = \Ga(c^\ast T\R^2) =
\left\{h: \quad \begin{aligned}\xymatrix{
& T\R^2 \ar[d]^{\pi} \\
S^1 \ar[r]^c \ar[ur]^h & \R^2
} \end{aligned} \right\}\,.
\end{equation*}
In our case, since the tangent bundle $T\R^2$ is trivial, it can also be identified with the space of $\R^2$-valued functions on $S^1$,
\[
T_c \on{Imm}(S^1, \R^2) \cong C^\infty(S^1, \R^2)\,.
\]

For a curve $c \in \on{Imm}(S^1, \R^2)$ we denote the parameter by $\th \in S^1$ and differentiation $\p_\th$ by $'$, i.e., $c'=\p_\th c$. Since $c$ is an immersion, the unit-length tangent vector $v = c' / |c'|$ is well-defined. Rotating $v$ by $\tfrac \pi 2$ we obtain the unit-length normal vector $n = Jv$, where $J$ is rotation by $\tfrac \pi 2$. We will denote by $D_s = \p_\th / |c_\th|$ the derivative with respect to arc-length and by $\ud s = |c_\th| \ud \th$ the integration with respect to arc-length. To summarize we have
\begin{align*}
v &= D_s c\,, & n &= Jv\,, & D_s &= \frac 1{|c_\th|} \p_\th\,, & \ud s &= |c_\th| \ud \th\,.
\end{align*}
The curvature can be defined as
\[
\ka = \langle D_s v, n \rangle
\]
and we have the Frenet-equations
\begin{align*}
D_s v &= \ka n \\
D_s n &= -\ka v\,.
\end{align*}
The length of a curve will be denoted by $\ell_c = \int_{S^1} 1 \ud s$.
We define the turning angle $\al : S^1 \to \R/2\pi\mathbb Z$ of a curve $c$ by $v(\th) = (\cos \al(\th), \sin \al(\th))$. Then curvature is given by $\ka = D_s \al$.

\subsection{Variational Formulae}\label{variational-formulae}

We will need formulas that express, how the quantities $v$, $n$ and $\ka$ change, if we vary the 
underlying curve $c$. For a smooth map $F$ from $\on{Imm}(S^1,\R^2)$ to any convenient vector space 
(see \cite{KM97})
we denote by  
\[
D_{c,h} F = \left.\frac{\ud}{\ud t}\right|_{t=0} F(c+th)
\]
the variation in the direction $h$.

The proof of the following formulas can be found for example in \cite{Michor2007}.
\begin{align*}
D_{c,h} v &= \langle D_s h, n \rangle n\quad\implies\quad D_{c,h}\al= \langle D_s h, n \rangle\\
D_{c,h} n &= -\langle D_s h, n \rangle v \\
D_{c,h} \ka &= \langle D_s^2 h, n \rangle - 2\ka \langle D_s h, v \rangle \\
D_{c,h}\left( |c'|^k \right) &= k\,\langle D_s h, v \rangle \,|c'|^k\,.
\end{align*}
With these basic building blocks, one can use the following lemma to compute the variations of higher derivatives.

\begin{lemma} 
\label{lem:Ds_rmap}
If $F$ is a smooth map $F : \on{Imm}(S^1,\R^2) \to C^\infty(S^1, \R^d)$, then the variation of the composition $D_s \o F$ is given by
\[
D_{c,h} \left(D_s \o F\right) = D_s \left(D_{c,h} F\right) - \langle D_s h, v \rangle D_s F(c)\,.
\]
\end{lemma}

\begin{proof}
The operator $\p_\th$ is linear and thus commutes with the derivative with respect to $c$. Thus we have
\begin{align*}
D_{c,h} \left(D_s \o F\right) &= D_{c,h} \left( |c'|\i \p_\th F(c) \right)\\
&= |c'|\i \p_\th\left(D_{c,h}F\right) + \left(D_{c,h} \,|c'|\i\right) \p_\th F(c) \\
&= D_s \left( D_{c,h} F \right) - \langle D_s h, v \rangle\, |c'|\i \p_\th F(c) \\
&= D_s \left( D_{c,h} F \right) - \langle D_s h, v \rangle\, D_s F(c) \,. \qedhere
\end{align*}
\end{proof}

\subsection{Sobolev Norms}

In this paper we will only consider Sobolev spaces of integer order. For $n\geq 1$ the $H^n(d\th)$-norm on $C^\infty(S^1,\R^d)$ is given by
\begin{equation}
\label{eq:Hn_dtheta}
\| u \|_{H^n(d\th)}^2 = \int_{S^1} |u|^2 + |\p_\th^n u|^2 \ud \th\,.
\end{equation}
Given $c \in \on{Imm}(S^1,\R^2)$, we define the $H^n(ds)$-norm on $C^\infty(S^1,\R^d)$ by
\begin{equation}
\label{eq:Hn_ds}
\| u \|_{H^n(ds)}^2 = \int_{S^1} |u(s)|^2 + |D_s^{n}u(s)|^2 \ud s\,.
\end{equation}
Note that in \eqref{eq:Hn_ds} integration and differentiation are performed with respect to the arc-length of $c$, while in \eqref{eq:Hn_dtheta} the parameter $\th$ is used. In particular the $H^n(ds)$-norm depends on the curve $c$. The norms $H^n(d\th)$ and $H^n(ds)$ are equivalent, but the constants do depend on $c$. We prove in Lem.\ \ref{lem:Hk_local_equivalence}, that if $c$ doesn't vary too much, the constants can be chosen independently of $c$.

The $L^2(d\th)$- and $L^2(ds)$-norms are defined similarly,
\begin{align*}
\| u \|^2_{L^2(d\th)} &= \int_{S^1} |u|^2 \ud \th\,, &
\| u \|^2_{L^2(ds)} = \int_{S^1} |u|^2 \ud s\,,
\end{align*}
and they are related via $\left\| u \sqrt{|c'|} \right\|_{L^2(d\th)} = \| u \|_{L^2(ds)}$. Whenever we write $H^n(S^1,\R^d)$ or $L^2(S^1,\R^d)$, we always endow them with the $H^n(d\th)$- and $L^2(d\th)$-norms.

For $n \geq 2$ we shall denote by
\[				
\on{Imm}^n(S^1,\R^2) = \{ c \,:\, c \in H^n(S^1,\R^2), c'(\th) \neq 0 \}
\]
the space of Sobolev immersions of order $n$. Because of the Sobolev embedding theorem, see 
\cite{Adams2003}, we have $H^2(S^1,\R^2) \hookrightarrow C^1(S^1,\R^2)$ and thus 
$\on{Imm}^n(S^1,\R^2)$ is well-defined. We will see in Sect.\ \ref{sec:Sobolev_immersions} that the 
$H^n(ds)$-norm remains well-defined if $c \in \on{Imm}^n(S^1,\R^2)$.   

The following result on point-wise multiplication will be used repeatedly. It can be found, among 
other places in \cite[Lem.\ 2.3]{Inci2013}. We will in particular use that $k$ can be negative. 

\begin{lemma}
\label{lem:sob_mult}
Let $n \geq 1$ and $k \in \mathbb Z$ with $|k|\leq n$ Then multiplication is a bounded bilinear map
\[
\cdot: H^{n}(S^1,\R^d) \x H^k(S^1,\R^d) \to H^k(S^1,\R)\,,\quad
(f, g) \mapsto \langle f, g \rangle
\]
\end{lemma}

The last tool, that we will need is composition of Sobolev diffeomorphisms. For $n \geq 1$, define
\[
\mc D^n(S^1) = \{ \ph \,:\, \ph \text{ is $C^1$-diffeomorphism of $S^1$ and } \ph \in H^n(S^1,S^1)\}
\]
the group of Sobolev diffeomorphisms. The following lemma can be found in \cite[Thm.\ 1.2]{Inci2013}.

\begin{lemma}\label{Sobolev-composition}
Let $n \geq 2$ and $0 \leq k \leq n$. Then the composition map
\[
H^k(S^1,\R^d) \x \mc D^n(S^1) \to H^k(S^1,\R^d)\,,\quad
(f, \ph) \mapsto f \o \ph
\]
is continuous.
\end{lemma}

Let $n\geq 2$ and fix $\ph \in \mc D^n(S^1)$. Denote by $R_\ph(h) = h \o \ph$ the composition with $\ph$. From Lem.\ \ref{Sobolev-composition} we see that $R_\ph$ is a bounded linear map $R_\ph : H^n \to H^n$. The following lemma tells us that the transpose of this map respects Sobolev orders.

\begin{lemma}\label{Sobolev-transpose}
Let $n \geq 2$, $\ph \in \mc D^n(S^1)$ and $-n \leq k \leq n-1$. Then the restrictions of $R_\ph^\ast$ are bounded linear maps
\[
R_\ph^\ast \upharpoonright H^k(S^1,\R^d) : H^k(S^1,\R^d) \to H^k(S^1,\R^d)\,.
\]
On $L^2(S^1,\R^d)$ we have the identity $R_{\ph\i}^\ast(f) = R_{\ph}(f)\, \ph'$.
\end{lemma}

\begin{proof}
For $-n \leq k \leq 0$, we obtain from Lem.\ \ref{Sobolev-composition} that $R_\ph$ is a map $R_\ph : H^{-k} \to H^{-k}$ and by $L^2$-duality we obtain that $R_\ph^\ast : H^k \to H^k$ as required.

Now let $0 \leq k \leq n-1$, $f \in H^k$ and $g \in H^n$. We replace $\ph$ by $\ph\i$ to simplify the formulas. By definition of the transpose
\begin{multline*}
\left\langle R_{\ph\i}^\ast f, g \right\rangle_{H^{-n}\x H^n}
= \left\langle f, R_{\ph\i}\, g \right\rangle_{H^{-n}\x H^n} = \\
= \int_{S^1} \left\langle f(\th),  g(\ph\i(\th)) \right\rangle\ud \th 
= \int_{S^1} \left\langle f(\ph(\th)), g(\th) \right\rangle \ph'(\th) \ud \th = \\
= \left\langle \left(R_\ph f\right) \ph', g\right\rangle_{H^{-p}\x H^p}\,.
\end{multline*}
Thus we obtain $R_{\ph\i}^\ast(f) = R_{\ph}(f)\, \ph'$ and using Lem.\ \ref{lem:sob_mult} we see that for $f \in H^k$ we also have $R_{\ph\i}^\ast(f) \in H^k$.
\end{proof}

\subsection{Notation}\label{notation:lesssim}

We will write
\[
f \lesssim_{A} g
\]
if there exists a constant $C > 0$, possibly depending on $A$, such that the inequality $f \leq C g$ holds.

\subsection{Gronwall Inequalities}

The following version of Gronwall's inequality can be found in \cite[Thm.\ 1.3.2]{Pachpatte1998} and \cite{Jones1964}.

\begin{theorem}
\label{thm:gronwall}
Let $A$, $\Ph$, $\Ps$ be real continuous functions defined on $[a,b]$ and $\Ph \geq 0$. We suppose that on $[a,b]$ we have the following inequality
\[
A(t) \leq \Ps(t) + \int_a^t A(s)\Ph(s) \ud s\,.
\]
Then
\[
A(t) \leq \Ps(t) + \int_a^t \Ps(s)\Ph(s) \exp\left(\int_s^t \Ph(u) \ud u \right) \ud s
\]
holds on $[a,b]$.
\end{theorem}

We will repeatedly use the following corollary.

\begin{corollary}
\label{cor:gronwall_applied}
Let $A$, $G$ be real, continuous functions on $[0,T]$  with $G\geq 0$ and $\al, \be$ non-negative constants. We suppose that on $[0,T]$ we have the inequality
\[
A(t) \leq A(0) + \int_0^t (\al + \be A(s)) G(s) \ud s\,.
\]
Then
\[
A(t) \leq A(0) + \left(\al + (A(0) + \al N) \be e^{\be N}\right) \int_0^t G(s) \ud s
\]
holds in $[0,T]$ with $N = \int_0^T G(t) \ud t$.
\end{corollary}

\begin{proof}
Apply the Gronwall inequality with $[a,b]=[0,T]$, $\Ps(t) = A(0) + \al \int_0^t G(s) \ud s$ and $\Ph(s) = \be G(s)$, and note that $G(s) \geq 0$ implies $\int_s^t G(u) \ud u \leq N$.
\end{proof}

\subsection{Poincar\'e Inequalities}

In the later sections it will be necessary to estimate the $H^k(ds)$-norm of a function by the $H^{n}(ds)$-norm with $k<n$, as well as the $L^\infty$-norm by the $H^k(ds)$-norm. In particular, we will need to know, how the curve $c$ enters into the estimates. The basic result is the following lemma, which is adapted from \cite[Lem.\ 18]{Mennucci2008}.

\begin{lemma}
\label{lem:poincare_basic}
Let $c \in \on{Imm}^2(S^1,\R^2)$ and $h:S^1\to \R^d$ be absolutely continuous. Then
\[
\sup_{\th \in S^1} \left| h(\th) - \frac 1{\ell_c} \int_{S^1} h \ud s \right| \leq \frac 12 \int_{S^1} |D_sh| \ud s\,.
\]
\end{lemma}

\begin{proof}
Since $h(0) = h(2\pi)$, the following equality holds,
\[
h(\th) - h(0) = \frac 12 \left( \int_0^\th h'(\si) \ud \si - \int_\th^{2\pi} h'(\si) \ud \si \right) \,,
\]
and hence after integration
\[
\frac 1{\ell_c} \int_{S^1} h \ud s - h(0) = \frac{1}{2\ell_c} \int_{S^1} \left( \int_0^\th h'(\si) \ud \si - \int_\th^{2\pi} h'(\si) \ud \si \right) \ud s\,.
\]
Next we take the absolute value
\begin{align*}
\left| \frac 1{\ell_c} \int_{S^1} h \ud s - h(0) \right| &\leq \frac{1}{2\ell_c} \int_{S^1} \left( \int_0^\th |h'(\si)| \ud \si + \int_\th^{2\pi} |h'(\si)| \ud \si \right) \ud s \\
&\leq \frac{1}{2\ell_c} \int_{S^1} |h'(\si)|\ud \si \int_{S^1} 1 \ud s = \frac 12 \int_{S^1} |D_sh| \ud s
\end{align*}
Now we replace 0 by an arbitrary $\th \in S^1$ and repeat the above steps.
 \end{proof}

This lemma permits us to prove the inequalities that we will use throughout the remainder of the paper.

\begin{lemma}
\label{lem:poincare}
Let $c \in \on{Imm}^2(S^1,\R^2)$ and $h \in H^2(S^1,\R^d)$. Then
\begin{itemize}
\item
$\| h\|_{L^\infty}^2 \leq \displaystyle\frac 2{\ell_c} \| h \|_{L^2(ds)}^2 + \displaystyle \frac {\ell_c} 2  \| D_s h \|_{L^2(ds)}^2\,,$
\item
$\| D_s h \|_{L^\infty}^2 \leq \displaystyle \frac {\ell_c}4 \| D_s^2 h \|_{L^2(ds)}^2\,,$
\item
$\| D_s h \|_{L^2(ds)}^2 \leq \displaystyle \frac {\ell_c^2}4 \| D_s^2 h \|_{L^2(ds)}^2\,.$
\end{itemize}
\end{lemma}

\begin{proof}
From Lem.\ \ref{lem:poincare_basic} we obtain the inequality
\[
\| h \|_{L^\infty} \leq \frac 1{\ell_c} \int_{S^1} |h| \ud s + \frac 12 \int_{S^1} |D_s h| \ud s\,.
\]
Next we use $(a+b)^2 \leq 2a^2 + 2b^2$ and Cauchy-Schwarz in
\begin{align*}
\| h \|^2_{L^\infty} &\leq \frac 2{\ell_c^2} \left(\int_{S^1} |h| \ud s\right)^2
+ \frac 12 \left( \int_{S^1} |D_s h| \ud s\right)^2 \\
&\leq \frac 2{\ell_c} \left( \int_{S^1} |h|^2 \ud s \right) 
+ \frac{\ell_c}{2} \left( \int_{S^1} |D_s h|^2 \ud s \right)\,,
\end{align*}
thus proving the first statement. To prove the second statement we note that $\int_{S^1} D_s h \ud s = 0$ and thus by Lem.\ \ref{lem:poincare_basic}
\[
\| D_s h \|_{L^\infty} \leq \frac 12 \int_{S^1} |D_s^2 h| \ud s\,.
\]
Hence
\[
\| D_s h \|^2_{L^\infty} \leq \frac 14  \left(\int_{S^1} |D_s^2 h| \ud s\right)^2
\leq \frac{\ell_c}{4} \| D_s^2 h\|^2_{L^2(ds)}\,.
\]
To prove the third statement we estimate
\[
\| D_s h\|^2_{L^2(ds)} \leq \| D_s h \|^2_{L^\infty} \int_{S^1} 1 \ud s
\leq \frac{\ell_c^2}4 \| D_s^2 h\|^2_{L^2(ds)}\,.
\]
This completes the proof.
\end{proof}

The next lemma allows us to estimate the $H^k(ds)$-norm using a combination of the $L^2(ds)$- and the $H^n(ds)$-norms, without introducing constants that depend on the curve.

\begin{lemma}
\label{lem:poincare2}
Let $n\geq 2$, $c \in \on{Imm}^n(S^1,\R^2)$ and $h \in H^n(S^1,\R^d)$. Then for $0 \leq k \leq n$,
\[
\| D_s^k h \|^2_{L^2(ds)} \leq \| h\|^2_{L^2(ds)} + \| D_s^n h \|^2_{L^2(ds)}\,.
\]
\end{lemma}

\begin{proof}
Let us write $D_c$ and $L^2(c)$ for $D_s$ and $L^2(ds)$ respectively to emphasize the dependence on the curve $c$. Since $\left\|D_c^k h\right\|_{L^2(c)} = \left\|D_{c\o\ph}^k (h\o\ph)\right\|_{L^2(c\o\ph)}$, we can assume that $c$ has a constant speed parametrization, i.e. $|c'| = \ell_c/2\pi$. The inequality we have to show is
\[
\int_{0}^{2\pi} \left( \frac{2\pi}{\ell_c}\right)^{2k-1} \big|h^{(k)}(\th)\big|^{2} \ud \th \leq
\int_{0}^{2\pi} \frac{\ell_c}{2\pi} \left|h(\th)\right|^2  + \left( \frac{2\pi}{\ell_c}\right)^{2n-1} \big|h^{(n)}(\th)\big|^{2} \ud \th\,.
\]
Let $\ph(x) = \frac{2\pi}{\ell_c} x$. After a change of variables this becomes
\begin{equation}
\label{eq:ineq1}
\int_{0}^{\ell_c} \big|(h\o\ph)^{(k)}(x)\big|^{2} \ud x \leq
\int_{0}^{\ell_c} \left|h\o\ph(x)\right|^2  + \big|(h\o\ph)^{(n)}(x)\big|^{2} \ud x\,.
\end{equation}
Let $f = h\o\ph$ and assume w.l.o.g. that $f$ is $\R$-valued. 
Define $f_k(x) = \ell_c^{-1/2} \exp\left(i\frac{2\pi k}{\ell_c}x\right)$, 
which is an orthonormal basis of $L^2([0,\ell_c], \R)$. 
Then $f = \sum_{k \in \mathbb Z} \wh f(k) f_k$ and \eqref{eq:ineq1} becomes
\[
\sum_{k \in \mathbb Z} \left(\tfrac {2\pi k}{\ell_c}\right)^{2k} \big| \wh f(k) \big|^2
\leq \sum_{k \in \mathbb Z} \left[ 1 + \left(\tfrac {2\pi k}{\ell_c}\right)^{2n}\right] \big| \wh f(k) \big|^2\,.
\]
Since for $a \geq 0$ we have the inequality $a^k \leq 1+a^n$, the last inequality is satisfied, thus concluding the proof.
\end{proof}

An alternative way to estimate the $H^k(ds)$-norm is given by the following lemma, which is the 
periodic version of the Gagliardo-Nirenberg inequalities (see \cite{Nirenberg1959}). 

\begin{lemma}
\label{lem:poincare3}
Let $n\geq 2$, $c \in \on{Imm}^n(S^1,\R^2)$ and $h \in H^n(S^1,\R^d)$. Then for $0 \leq k \leq n$,
\[
\| D_s^k h \|_{L^2(ds)} \leq \| h\|^{1-k/n}_{L^2(ds)} \, \| D_s^n h \|^{k/n}_{L^2(ds)}\,.
\]

If $c \in \on{Imm}^2(S^1,\R^2)$, the inequality also holds for $n=0,1$.
\end{lemma}

\subsection{The Geodesic Equation on Weak Riemannian Manifolds}

Let $V$ be a convenient vector space, $M \subseteq V$ an open subset and $G$ a possibly weak Riemannian metric on $M$. Denote by $\bar{L} : TM \to (TM)'$ the canonical map defined by
\[
G_c(h,k) = \langle \bar{L}_c h, k \rangle_{TM}\,,
\]
with $c \in M$, $h,k \in T_c M$ and with $\langle \cdot,\cdot \rangle_{TM}$ denoting the canonical pairing between $(TM)'$ and $TM$. We also define $H_c(h,h) \in (T_cM)'$ via
\[
D_{c,m} G_c(h,h) = \langle H_c(h,h), m \rangle_{TM}\,,
\]
with $D_{c,m}$ denoting the directional derivative at $c$ in direction $m$. 
In fact $H$ is a smooth map
\[
H : TM \to (TM)'\,,\quad (c,h) \mapsto (c, H_c(h,h))\,.
\]
With these definitions we can state how to calculate the geodesic equation.

\begin{lemma}
\label{lem:convenient_geod_eq}
The geodesic equation -- or equivalently the Levi-Civita covariant derivative -- on $(M,G)$ exists if and only if $\tfrac 12 H_c(h,h) - \left(D_{c,h} \bar{L}_c\right)\!(h)$ is in the image of $\bar{L}_c$ for all $(c,h) \in TM$ and the map
\[
TM \to TM\,,\quad (c,h) \mapsto \bar{L}_c\i \left(\tfrac 12 H_c(h,h) - \left(D_{c,h} \bar{L}_c\right)\!(h)\right)
\]
is smooth. In this case the geodesic equation can be written as
\[
\begin{aligned}
c_t &= \bar{L}_c\i p \\
p_t &= \frac12 H_c(c_t,c_t)
\end{aligned}
\qquad\text{ or }\qquad
c_{tt} = \frac12 \bar L_c\i\left(H_c(c_t,c_t) - \left(\p_t\bar L_c\right)\!(c_t)\right)\,.
\]
\end{lemma}

This lemma is an adaptation of the result given in \cite[2.4.1]{Bauer2013d_preprint} and the same proof can be repeated; see also \cite[Sect.\ 2.4]{Micheli2013}.

\section{Sobolev Metrics with Constant Coefficients}

In this paper we will consider \emph{Sobolev-type metrics} with constant coefficients. These are metrics of the form
\[
G_c(h,k) = \int_{S^1} \sum_{j=0}^n a_j \langle D_s^j h, D_s^j k \rangle \ud s\,,
\]
with $a_j\geq 0$ and $a_0, a_n \neq 0$. We call $n$ the \emph{order} of the metric. The metric can be defined either on the space $\on{Imm}(S^1,\R^2)$ of ($C^\infty$-)smooth immersions or for $p \geq n$ on the spaces $\on{Imm}^p(S^1,\R^2)$ of Sobolev $H^p$-immersions.

\subsection{The Space of Smooth Immersions}

Let us first consider $G$ on the space of smooth immersions. The metric can be represented via the associated family of operators, $L$, which are defined by
\[
G_c(h,k) = \int_{S^1} \langle L_c h, k \rangle \ud s = \int_{S^1} \langle h, L_c k \rangle \ud s\,,
\]
The operator $L_c : T_c \on{Imm}(S^1,\R^2) \to T_c \on{Imm}(S^1,\R^2)$ for a Sobolev metric with constant coefficients can be calculated via integration by parts and is given by
\[
L_c h = \sum_{j=0}^n (-1)^j a_j D_s^{2j} h\,.
\]
The operator $L_c$ is self-adjoint, positive and hence injective. Since $L_c$ is elliptic, it is 
Fredholm $H^k\to H^{k-2n}$ with vanishing index and thus surjective. Furthermore its inverse is smooth as well. 
We want to distinguish between the operator $L_c$ and the canonical embedding from $T_c\on{Imm}$ into $(T_c \on{Imm})'$, 
which we denote by $\bar{L}_c$. They are related via
\[
\bar{L}_c h = L_c h \otimes \ud s = L_c h \otimes |c'| \ud \th\,.
\]
Later we will simply write $\bar{L}_c h = L_c h\,|c'|$, especially when the order of multiplication and differentiation becomes important in Sobolev spaces.

\subsection{The Space of Sobolev Immersions}
\label{sec:Sobolev_immersions}

Assume $n \geq 2$ and let $G$ be a Sobolev metric of order $n$. We want to extend $G$ from the 
space $\on{Imm}(S^1,\R^2)$ to a smooth metric on the Sobolev-completion $\on{Imm}^n(S^1,\R^2)$. 
First we have to look at the action of the arc-length derivative and its transpose (with respect to 
$H^0(d\th)$) on Sobolev 
spaces. Remember that we always use the $H^n(d\th)$-norm on Sobolev completions.   
We can write $D_s$ as the composition $D_s = \tfrac{1}{|c'|} \o \p_\th$, where $\tfrac{1}{|c'|}$ is 
interpreted as the multiplication operator $f \mapsto \tfrac{1}{|c'|} f$. Its transpose is 
$D_s^\ast = \p_\th^\ast \o \left(\tfrac{1}{|c'|}\right)^\ast = -\p_\th \o \tfrac{1}{|c'|}$. These 
operators are smooth in the following sense.   

\begin{lemma}
\label{lem:Ds_smooth_sobolev}
Let $n\geq 2$ and $k \in \mathbb Z$ with $|k| \leq n-1$. Then the maps
\begin{align*}
D_s &: \on{Imm}^n(S^1,\R^2) \x H^{k+1}(S^1,\R^d) \to H^k(S^1,\R^d)\,,\quad
(c,h) \mapsto D_s h = \tfrac{1}{|c'|} h' \\
D_s^\ast &: \on{Imm}^n(S^1,\R^2) \x H^{k}(S^1,\R^d) \to H^{k-1}(S^1,\R^d)\,,\quad
(c, h)  \mapsto  D_s^\ast h = -\left(\tfrac 1{|c'|} h\right)' 
\end{align*}
are smooth.
\end{lemma}

\begin{proof}
For $n\geq 2$, the map $c \mapsto \tfrac{1}{|c'|}$ is the composition of the following smooth maps,
\[
\begin{array}{ccccc}
\on{Imm}^n(S^1,\R^2) & \to &
\{ f : f > 0 \} \subset H^{n-1}(S^1,\R) & \to &
H^{n-1}(S^1,\R) \\
c & \mapsto & |c'| & \mapsto & \tfrac{1}{|c'|}
\end{array}\,.
\]
Since $\tfrac{1}{|c'|} \in H^{n-1}(S^1,\R^2)$, Lem.\ \ref{lem:sob_mult}. concludes the proof.
\end{proof}

Using Lem.\ \ref{lem:Ds_smooth_sobolev} we see that
\[
G_c(h,h) = \int_{S^1} \sum_{k=0}^n a_k \langle D_s^k h, D_s^k h \rangle \ud s
\]
is well-defined for $(c,h) \in T\on{Imm}^n(S^1,\R^2)$. As the tangent bundle is isomorphic to $T\on{Imm}^n(S^1,\R^2) \cong \on{Imm}^n(S^1,\R^2) \x H^n(S^1,\R^2)$, we can also write the metric as
\[
G_c(h,h) = \left\langle \sum_{k=0}^n a_k\, (D_s^k)^\ast\, |c'|\, D_s^k h, h \right\rangle_{H^{-n}\x H^n}\,.
\]
Again we note that $|c'|$ has to be interpreted as the multiplication operator $f \mapsto |c'| \,f$ on the spaces $H^k$ with $|k| \leq n-1$. Thus the operator $\bar{L}_c:H^n \to H^{-n}$ is given by
\begin{equation*}
\bar{L}_c = \sum_{k=0}^n a_k\, (D_s^k)^\ast \o |c'| \o  D_s^k\,.
\end{equation*}
While it is tempting to ``simplify'' the expression for $\bar{L}_c$ using the identity
\[
D_s^\ast \o |c'| = -|c'| \o D_s\,,
\]
one has to be careful, since the identity is only valid, when interpreted as an operator $H^k \to H^{k-1}$ with $-n+2 \leq k \leq n-1$. The left hand side however makes sense also for $k=-n+1$. Thus we have the operator
\[
(D_s^n)^\ast \o |c'| : L^2 \to H^{-n}\,,
\]
but the domain has to be at least $H^1$ for the operator
\[
(-1)^n\, |c'| \o D_s^n : H^1 \to H^{-n+1}\,.
\]
So the expression
\[
\bar{L}_c h = \sum_{k=0}^n (-1)^k a_k\, |c'|\,D_s^{2k} h \,,
\]
is only valid, when we restrict $\bar{L}_c$ to $H^{n+1}$, i.e., $\bar{L}_c : H^{n+1} \to H^{-n+1}$.

\subsection{The Geodesic Equation}

By Lem.\ \ref{lem:convenient_geod_eq}, we need to calculate $H_c(h,h)$. This is achieved in the following lemma.

\begin{lemma}
Let $n \geq 2$ and let $G$ be a Sobolev metric of order $n$. On $\on{Imm}^n(S^1,\R^2)$ we have
\begin{equation}
\label{eq:H_gradient_Sobolev}
H_c(h,h) = -a_0\, |c'|\, D_s\left( \langle h, h \rangle v \right)
- \sum_{k=1}^n \sum_{j=1}^{2k-1} (-1)^{k+j} a_k\, D_s^\ast \o
\left( |c'| \langle D_s^{2k-j} h, D_s^j h \rangle v \right)\,.
\end{equation}
On $\on{Imm}^p(S^1,\R^2)$ with $p \geq n+1$ as well as $\on{Imm}(S^1,\R^2)$ we have the equivalent expression,
\begin{align*}
H_c(h,h) &=\Bigg( -2 \langle L_c h, D_s h \rangle v - a_0 \langle h, h \rangle \ka n+{}\\
&\qquad{}+ \sum_{k=1}^n \sum_{j=1}^{2k-1} (-1)^{k+j} a_k \langle D_s^{2k-j} h, D_s^j h \rangle \ka n \Bigg) \otimes \ud s\,.
\end{align*}
\end{lemma}

\begin{proof}
For $k\geq 1$ the variation of the $k$-th arc-length derivative is
\[
D_{c,m} D_s^k h = - \sum_{j=1}^k D_s^{k-j} \left( \langle D_s m, v \rangle D_s^j h \right)\,,
\]
and the formula is valid for $(c,m) \in T\on{Imm}^n(S^1,\R^2)$ and $h \in H^{-n+k}(S^1,\R^d)$. So
\begin{align*}
D_{c,m} G_c(h,h) 
&=\!\! \int_{S^1} \sum_{k=0}^n a_k\, \langle D_s^k h, D_s^k h \rangle
\langle D_s m, v \rangle \, |c'|
+ 2 \sum_{k=1}^n a_k \left\langle D_s^k h, D_{c,m} D_s^k h \right\rangle |c'| \ud \th \\
&= \left\langle \sum_{k=0}^n a_k \,|c'| \langle D_s^k h, D_s^k \rangle v, 
D_s m \right\rangle_{H^{-n+1}\x H^{n-1}} \\
&\qquad\qquad -2 \sum_{k=1}^n \sum_{j=1}^k a_k\, \left\langle
|c'| D_s^k h, D_s^{k-j} \langle D_s m, v \rangle D_s^j h \right\rangle_{H^{-n+k}\x H^{n-k}}\,.
\end{align*}
Each term in the second sum is equal to
\begin{align*}
\Big\langle |c'| D_s^k h, 
D_s^{k-j} \langle D_s m, v \rangle &D_s^j h \Big\rangle_{H^{-n+k}\x H^{n-k}} = \\
&= \left\langle \left(D_s^{k-j}\right)^\ast |c'|\, D_s^k h, 
\langle D_s m, v \rangle D_s^j h \right\rangle_{H^{-n+j}\x H^{n-j}} \\
&= (-1)^{k-j} \left\langle |c'|\, D_s^{2k-j} h, 
\langle D_s m, v \rangle D_s^j h \right\rangle_{H^{-n+j}\x H^{n-j}} \\
&= (-1)^{k-j} \left\langle |c'|\, \langle D_s^{2k-j} h, 
D_s^j h \rangle v, D_s m \right\rangle_{H^{-n+1}\x H^{n-1}}\,.
\end{align*}
So
\begin{align*}
H_c&(h,h) =\\
&=
\sum_{k=0}^n a_k D_s^\ast \o \left( |c'| \langle D_s^k h, D_s^k h \rangle v \right)
- 2\sum_{k=1}^n \sum_{j=1}^k (-1)^{k-j} a_k D_s^\ast \o \left(
|c'|\langle D_s^{2k-j} h, D_s^j h \rangle v \right) \\
&= -a_0\, |c'| D_s\left( \langle h, h \rangle v \right)
- \sum_{k=1}^n \sum_{j=1}^{2k-1} (-1)^{k+j} a_k\, D_s^\ast \o
\left( |c'| \langle D_s^{2k-j} h, D_s^j h \rangle v \right)\,.
\end{align*}
This proves the first formula.

If $(c,h) \in T\on{Imm}^p(S^1,\R^2)$ with $p \geq 1$, we can commute $D_s^\ast \o |c'| = -|c'| \o D_s$ to obtain
\[
H_c(h,h)
= -a_0\, |c'| D_s\left( \langle h, h \rangle v \right)
+ \sum_{k=1}^n \sum_{j=1}^{2k-1} (-1)^{k+j} a_k\, |c'|\, D_s \left( \langle D_s^{2k-j} h, D_s^j h \rangle v \right)\,.
\]
Parts of the expression simplify as follows
\begin{align*}
&\sum_{k=1}^n \sum_{j=1}^{2k-1} (-1)^{k+j} a_k 
D_s \left( \langle D_s^{2k-j} h, D_s^j h \rangle \right) 
- a_0 D_s \Big( \langle h, h \rangle \Big) \\
&= \sum_{k=1}^n \sum_{j=1}^{2k-1} (-1)^{k+j} a_k
\left( \langle D_s^{2k-j+1} h, D_s^j h \rangle 
+ \langle D_s^{2k-j} h, D_s^{j+1} h \rangle \right) 
- 2a_0 \langle h, D_s h \rangle \\
&= \sum_{k=1}^n a_k \!\left( \sum_{j=0}^{2k-2} (-1)^{k+j+1}
\langle D_s^{2k-j} h, D_s^j h \rangle 
+ \sum_{j=1}^{2k-1} (-1)^{k+j}
\langle D_s^{2k-j} h, D_s^{j+1} h \rangle \right) 
\!\!-\! 2a_0 \langle h, D_s h\rangle \\
&= \sum_{k=1}^n (-1)^{k+1} 2a_k \langle D_s^{2k} h, D_s h \rangle 
- 2a_0 \langle h, D_s h\rangle \\
&= -2 \langle L_c h, D_s h \rangle\,,
\end{align*}
And by collecting the remaining terms we arrive at the desired result.
\end{proof}

Now that we have computed $H_c(h,h)$, we can write the geodesic equation of the metric $G$. It is
\begin{equation}
\label{eq:geod_eq}
\begin{split}
\p_t \left(\bar{L}_c c_t\right) &=
-\frac{a_0}2 \,|c'|\, D_s\left( \langle c_t, c_t \rangle v \right) \\
&\qquad{} - \sum_{k=1}^n \sum_{j=1}^{2k-1} (-1)^{k+j} \frac{a_k}{2}\, D_s^\ast \o
\left( |c'| \langle D_s^{2k-j} c_t, D_s^j c_t \rangle v \right)\,.
\end{split}
\end{equation}

\subsection{Local Well-Posedness}

It has been shown in \cite[Thm.\ 4.3]{Michor2007} that the geodesic equation of a Sobolev metric is 
well-posed on $\on{Imm}^p(S^1,\R^2)$ for $p \geq 2n+1$. For a metric of order $n\geq 2$ we extend 
the result to $p \geq n$. This will later simplify the proof of geodesic completeness.  

\begin{theorem}
\label{thm:geod_ex}
Let $n\geq 2$, $p\geq n$ and let $G$ be a Sobolev metric of order $n$ with constant 
coefficients. Then the geodesic equation \eqref{eq:geod_eq} has unique local solutions in the space 
$\on{Imm}^{p}(S^1,\R^2)$ of Sobolev $H^p$-immersions. The solutions depend $C^\infty$ on $t$ and 
the initial conditions. The domain of existence (in $t$) is uniform in $p$ and thus the geodesic 
equation also has local solutions in $\on{Imm}(S^1,\R^2)$, the space of smooth immersions.    
\end{theorem}

\begin{proof}
Fix $p \geq n$. For the geodesic equation to exist, we need to verify the assumptions in Lem.\ \ref{lem:convenient_geod_eq}. We first note that $\bar{L}_c$ is a map $\bar L_c : H^p \to H^{p-2n}$. By inspecting \eqref{eq:H_gradient_Sobolev} we see that $H_c(h,h) \in H^{p-2n}$ as well. Thus it remains to show that $\bar L_c$ maps $H^p$ onto $H^{p-2n}$ and that the inverse is smooth. This is shown in \ref{lem:Lc_inv_smooth}.

Regarding local existence, we rewrite the geodesic equation as a differential equation on $T\on{Imm}^n(S^1,\R^2)$,
\begin{align*}
c_t &= u \\
u_t &= \frac12 \bar L_c\i\left(H_c(u,u) - \left(D_{c,u} \bar L_c\right)\!(u)\right)\,.
\end{align*}
This is a smooth ODE on a Hilbert space and therefore by Picard-Lindel\"of it has local solutions, 
that depend smoothly on $t$ and the initial conditions. That the intervals of existence are uniform 
in the Sobolev order $p$, can be found in \cite[App. A]{Bauer2013d_preprint}. The result goes back 
to \cite[Thm.\ 12.1]{Ebin1970} and a different proof can be found in \cite{Michor2007}. 
\end{proof}

The following lemma shows that the operator $\bar L_c$ has a smooth inverse on appropriate Sobolev spaces. For $p=n$, we can use Lem. \ref{lem:Hk_local_equivalence} and the lemma of Lax-Milgram to show that $\bar L_c : H^n \to H^{-n}$ is invertible. For $p>n$ more work is necessary. Although $\bar L_c$ is an elliptic, positive differential operator, it has non-smooth coefficients. In fact, since $|c'| \in H^{n-1}$, some of the coefficients are only distributions. To overcome this, we will exploit the reparametrization invariance of the metric to transform $\bar L_c$ into a differential operator with constant coefficients.

\begin{lemma}
\label{lem:Lc_inv_smooth}
Let $n\geq 2$ and $G$ be a Sobolev metric of order $n$. For $p \geq n$ and $c \in \on{Imm}^p(S^1,\R^2)$, the associated operators
\[
\bar L_c : H^p(S^1,\R^d) \to H^{p-2n}(S^1,\R^d)\,,
\]
are isomorphisms and the map
\[
\bar L\i : \on{Imm}^p(S^1,\R^2) \x H^{p-2n}(S^1,\R^d) \to H^{p}(S^1,\R^d)\,,\quad
(c,h) \mapsto \bar L_c\i h
\]
is smooth.
\end{lemma}

\begin{proof}
Given a curve $c \in \on{Imm}^p(S^1,\R^2)$, we can write it as $c = d \o \ps$, where $d$ has constant speed, $|d'| = \ell_c / 2\pi$, and $\ps$ is a diffeomorphism of $S^1$. The pair $(d, \ps)$ is determined only up to rotations; we can remove the ambiguity by requiring that $c(0) = d(0)$. Then $\ps$ is given by
\[
\ps(\th) = \frac{2\pi}{\ell_c} \int_0^\th |c'(\si)| \ud \si\,.
\]
Concerning regularity, we have $\ps$ and $\ps^{-1} \in H^p(S^1,S^1)$ thus $\ps\in\mc D^p(S^1)$, and $d \in \on{Imm}^p(S^1,\R^2)$.

The reparametrization invariance of the metric $G$ implies
\[
\langle \bar{L}_c h, m \rangle_{H^{-p}\x H^p} = 
\left\langle \bar{L}_{c \o \ps\i}(h \o \ps\i), m \o\ps\i \right\rangle_{H^{-p}\x H^p}\,.
\]
Introduce the notation $R_\ph(h) = h \o \ph$. If $\ph \in \mc D^p(S^1)$ is a diffeomorphism, 
the map $R_\ph$ is an invertible linear map $R_\ph : H^p \to H^p$, by Lem.\ \ref{Sobolev-composition}. Furthermore by Lem.\ \ref{Sobolev-transpose} the transpose $R_\ph^\ast$ is an invertible map $R_\ph^\ast : H^{p-2n} \to H^{p-2n}$.
Thus we get
\[
\bar{L}_c h = R_{\ps\i}^\ast \o \bar{L}_d \o R_{\ps\i}(h)\,.
\]
Because $|d'|=\ell_c/2\pi$, the operator $\bar{L}_d$ is equal to
\[
\bar L_d = \sum_{k=0}^n (-1)^k a_k \left( \frac{2\pi}{\ell_c}\right)^{2k-1} \p_\th^{2k}\,.
\]
This is a positive, elliptic differential operator with constant coefficients and thus $\bar{L}_d : H^p \to H^{p-2n}$ is invertible. 
Thus the composition $\bar{L}_c : H^p \to H^{p-2n}$ is invertible. 
Smoothness of $(c,h) \mapsto \bar{L}_c\i h$ follows from the smoothness of $(c,h) \mapsto \bar{L}_c h$ 
and the implicit function theorem on Banach spaces.
\end{proof}

The remainder of the paper will be concerned with the analysis of the geodesic distance function induced by Sobolev metrics. These results will be used to show that geodesics for metrics of order $2$ and higher exist for all times.

\section{Lower Bounds on the Geodesic Distance}

To prepare the proof of geodesic completeness we first need to use geodesic distance to estimate quantities, that are derived from the curve and that appear in the geodesic equation. These include the length $\ell_c$, curvature $\ka$, its derivatives $D_s^k \ka$ as well as the length element $|c'|$ and its derivatives $D_s^k \log |c'|$. We want to show that they are bounded on metric balls of a Sobolev metric of sufficiently high order.

We start with the length $\ell_c$. The argument given in \cite[Sect.\ 4.7]{Michor2007} can be used to show the following slightly stronger statement.

\begin{lemma}
\label{lem:ell_lip}
Let the metric $G$ on $\on{Imm}(S^1,\R^2)$ satisfy
\[
\int_{S^1} \langle D_s h, v \rangle^2 \ud s \leq A\, G_c(h,h)\,,
\]
for some $A>0$. Then we have the estimate
\[
\left\| \sqrt{|c_1'|} - \sqrt{|c_2'|} \right\|_{L^2(d\th)} \leq \frac {\sqrt{A}}2 \on{dist}^G(c_1, c_2)\,,
\]
in particular the function 
$\sqrt{|c'|} : (\on{Imm}(S^1,\R^2), \on{dist}^G) \to L^2(S^1,\R)$
is Lipschitz.
\end{lemma}

\begin{proof}
Take two curves $c_1, c_2 \in \on{Imm}(S^1,\R^2)$ and let $c(t,\th)$ be a smooth path between them. Then the following relation holds pointwise for each $\th \in S^1$,
\[
\sqrt{|c_2'|}(\th) - \sqrt{|c_1'|}(\th) = \int_0^1 \p_t \left(\sqrt{|c'|} \right)(t,\th) \ud t\,.
\]
The derivative $\p_t \sqrt{|c'|}$ is given by
\[
\p_t \sqrt{|c'|} = \frac 12 \langle D_s c_t,  v\rangle \sqrt{|c'|}\,,
\]
and so
\begin{align*}
\left\| \sqrt{|c_1'|} - \sqrt{|c_2'|} \right\|_{L^2(d\th)} 
&\leq \frac 12 \int_0^1 \left\| \langle D_s c_t, v \rangle 
\sqrt{|c'|} \right\|_{L^2(d\th)} \ud t \\
&\leq \frac 12 \int_0^1 \big\| \langle D_sc_t,v \rangle \big\|_{L^2(ds)} \ud t \\
&\leq \frac {\sqrt{A}}2 \int_0^1 \sqrt{ G_c(c_t,c_t)} \ud t \\
&\leq \frac {\sqrt{A}}2 \on{Len}^G(c)\,.
\end{align*}
Since this estimate holds for every smooth path $c$, by taking the infimum we obtain 
\[
\left\| \sqrt{|c_1'|} - \sqrt{|c_2'|} \right\|_{L^2} \leq \frac {\sqrt{A}}2 \inf_{c} \on{Len}^G(c) 
= \frac {\sqrt{A}}2 \on{dist}^G(c_1,c_2)\,.	 \qedhere
\]
\end{proof}

We recover the statement of \cite[Sect.\ 4.7]{Michor2007} by applying the reverse triangle 
inequality. The following corollary is a disguised version of the fact, that on a normed space the 
norm function is Lipschitz.  

\begin{corollary}
\label{cor:ell_lip}
If the metric $G$ on $\on{Imm}(S^1,\R^2)$ satisfies
\[
\int_{S^1} \langle D_s h, v \rangle^2 \ud s \leq A\, G_c(h,h)\,,
\]
for some $A>0$, then the function $\sqrt{\ell_c} : (\on{Imm}(S^1,\R^2), \on{dist}^G) \to \R_{>0}$ is Lipschitz.
\end{corollary}

\begin{proof}
The statement follows from
\[
\ell_c = \int_{S^1} |c'(\th)| \ud \th = \left\| \sqrt{|c'|} \right\|_{L^2(d\th)}^2\,,
\]
and the inequality
\begin{align}
\label{eq:ell_lip_const}
\left|\sqrt{\ell_{c_1}} - \sqrt{\ell_{c_2}} \right| 
&= \left| \left\| \sqrt{|c'_1|} \right\|_{L^2(d\th)} - \left\| \sqrt{|c'_2|} \right\|_{L^2(d\th)} \right| 
\\\notag
&\leq \left\| \sqrt{|c_1'|} - \sqrt{|c_2'|} \right\|_{L^2(d\th)} \leq \frac {\sqrt{A}}2 
\on{dist}^G(c_1, c_2)\,. \qedhere
\end{align}
\end{proof}

\begin{remark}
Lemma \ref{lem:ell_lip} and Cor.\ \ref{cor:ell_lip} apply in particular to Sobolev metrics of order $n\geq 1$. 
For $n=1$ this is clear from $\langle D_s h, v \rangle^2 \leq |D_s h|^2$. 
For $n\geq 2$ we use Lem.\ \ref{lem:poincare2} to estimate
\[
\int_{S^1} \langle D_s h, v \rangle^2 \ud s \leq \| D_s h \|^2_{L^2(ds)}
\leq  \| h\|^2_{L^2(ds)} + \| D_s^n h \|^2_{L^2(ds)}
\leq \max\left(a_0^{-1}, a_n^{-1}\right) G_c(h,h)\,.
\]
We could have also used Lem.\ \ref{lem:poincare3},
\[
\int_{S^1} \langle D_s h, v \rangle^2 \ud s \leq \| D_s h \|^2_{L^2(ds)}
\leq \| h \|^{2-2/n}_{L^2(ds)}\, \| D_s^n h \|^{2/n}_{L^2(ds)}
\leq a_0^{(1-n)/n} a_n^{-1/n} G_c(h,h)\,,
\]
to reach the same conclusion.
\end{remark}

The following lemma shows a similar statement for $\ell_c^{-1/2}$. We do not get global Lipschitz continuity, instead the function $\ell_c^{-1/2}$ is Lipschitz on every metric ball. This implies that $\ell_c\i$ is bounded on every metric ball. We will show later in Cor.\ \ref{cor:lenpt_inv_bound} that the pointwise quantities $|c'(\th)|$ and $|c'(\th)|\i$ are also bounded on metric balls.

\begin{lemma}
\label{lem:ellc_inv_lipschitz}
Let the metric $G$ on $\on{Imm}(S^1,\R^2)$ satisfy
\[
\int_{S^1} \left| h \right|^2 + \left|D_s^n h \right|^2 \ud s \leq A\, G_c(h,h)
\]
for some $n\geq 2$ and some $A>0$. Given $c_0 \in \on{Imm}(S^1,\R^2)$ and $N>0$ there exists a 
constant $C = C(c_0, N)$ such that for all $c_1, c_2 \in \on{Imm}(S^1,\R^2)$ with $\on{dist}^G(c_0, c_i) < N$, 
$i=1,2$, we have
\[
\left| \ell_{c_1}^{-1/2} - \ell_{c_2}^{-1/2} \right| < C(c_0,N)\, \on{dist}^G(c_1,c_2)\,.
\]
In particular the function $\ell_c^{-1/2} : (\on{Imm}(S^1,\R^2), \on{dist}^G) \to \R_{>0}$ 
is Lipschitz on every metric ball.
\end{lemma}

\begin{proof}
Fix $c_1, c_2$ with $\on{dist}^G(c_0, c_i) < N$ and let $c(t,\th)$ be a path between them, 
such that $\on{dist}^G(c_0, c(t)) < 2N$. Then
\begin{align*}
\p_t \left( \ell_c^{-1/2} \right) &= - \tfrac 12 \ell_c^{-3/2} \int_{S^1} \langle D_s c_t, v \rangle\, |c'| \ud \th\,,
\end{align*}
and by taking absolute values
\begin{align*}
\left| \p_t \left( \ell_c^{-1/2} \right) \right|&\leq \tfrac 12 \ell_c^{-3/2} \int_{S^1} \left| \langle D_s c_t, v \rangle\right| \, |c'| \ud \th \\
&\leq \tfrac 12 \ell_c^{-3/2} \sqrt{\int_{S^1} |c'| \ud \th} \sqrt{\int_{S^1} \langle D_s c_t, v \rangle^2\, |c'| \ud \th} \\
&\leq \tfrac 12 \ell_c^{-1} \| D_s c_t \|_{L^2(ds)} 
\leq \tfrac 12 \ell_c^{-1} \left( 
\frac{\ell_c}{2} \right)^{n-1} \| D_s^n c_t \|_{L^2(ds)}
\quad\text{  by \ref{lem:poincare},} \\
&\leq 2^{-n}\, \ell_c^{n-2}\, \sqrt{A}\, \sqrt{G_c(c_t, c_t)} \,.
\end{align*}
By Cor.\ \ref{cor:ell_lip} the length $\ell_c$ is bounded along the path $c(t,\th)$ and and since $n\geq 2$ so is $\ell_c^{n-2}$. Thus
\begin{align*}
\left| \ell_{c_1}^{-1/2} - \ell_{c_2}^{-1/2} \right|
&\leq \int_{0}^1 \left| \p_t \left( \ell_c^{-1/2} \right) \right| \ud t \\
&\leq 2^{-n}\sqrt{A} \int_0^1 \ell_c^{n-2} \sqrt{G_c(c_t,c_t)} \ud t \\
&\lesssim_{c_0,N} \on{Len}^G(c)\,;\quad\text{  see \ref{notation:lesssim} for notation.}
\end{align*}
After taking the infimum over all paths connecting $c_1$ and $c_2$ we obtain
\[
\left| \ell_{c_1}^{-1/2} - \ell_{c_2}^{-1/2} \right| \lesssim_{c_0,N} \on{dist}^G(c_1, c_2)\,. 
	\qedhere
\]
\end{proof}

\begin{remark*}
We can compute the constant $C=C(c_0,N)$ in Lem.\ \ref{lem:ellc_inv_lipschitz} explicitly. Indeed from
\[
\left| \ell_{c_1}^{-1/2} - \ell_{c_2}^{-1/2} \right|
\leq 2^{-n}\sqrt{A} \int_0^1 \ell_c^{n-2} \sqrt{G_c(c_t,c_t)} \ud t\,,
\]
we obtain, following the proof,
\[
\left| \ell_{c_1}^{-1/2} - \ell_{c_2}^{-1/2} \right|
\leq 2^{-n}\sqrt{A} \left( \sup_{\on{dist}^G(c,c_0) < N} \ell_c^{n-2} \right)
\on{dist}^G(c_1,c_2)\,.
\]
Now, using \eqref{eq:ell_lip_const}, we can estimate $\ell_c$ via
\[
\sqrt{\ell_c} \leq 
\sqrt{\ell_{c_0}} + \left| \sqrt{\ell_c} - \sqrt{\ell_{c_0}} \right| \leq
\sqrt{\ell_{c_0}} + \frac 12 \sqrt{A} \on{dist}^G(c,c_0) \leq
\sqrt{\ell_{c_0}} + \frac 12 \sqrt{A}N\,.
\]
Thus we can use
\[
C(c_0, N) = 2^{-n} \sqrt{A} \left(\sqrt{\ell_{c_0}} + \tfrac 12 \sqrt{A}N\right)^{2n-4}
\]
for the constant.
\end{remark*}

\begin{corollary}
\label{cor:ellc_inv_bounded}
Let $G$ satisfy the assumptions of Lem.\ \ref{lem:ellc_inv_lipschitz}. Then $\ell_c^{-1}$ is bounded on every metric ball of $(\on{Imm}(S^1,\R^2), \on{dist}^G)$.
\end{corollary}

\begin{proof}
Fix $c_0 \in \on{Imm}(S^1,\R^2)$ and $N>0$ and let $c \in \on{Imm}(S^1,\R^2)$ with $\on{dist}^G(c_0, c) < N$. Then
\[
\ell_c^{-1/2} \leq \ell_{c_0}^{-1/2} + \left| \ell_{c_0}^{-1/2} - \ell_{c}^{-1/2} \right| \lesssim_{c_0,N} \ell_{c_0} + \on{dist}^G(c_0, c) \lesssim_{c_0, N} 1\,,
\]
and thus $\ell_c^{-1/2}$ is bounded on metric balls, which implies that $\ell_c\i$ is bounded as well.
\end{proof}

The variations of the turning angle $\al$ and of $\log |c'|$ are given by
\begin{align*}
D_{c,h} \left(\log |c'| \right) &= \langle D_s h, v \rangle \\
D_{c,h} \al &= \langle D_s h, n \rangle\,.
\end{align*}
As a preparation for the proof of Thm.\ \ref{thm:ka_bound} we compute explicit expressions for the variations of their derivatives.

\begin{lemma}
\label{lem:ka_log_var}
Let $c \in \on{Imm}(S^1, \R^2)$, $h \in T_c \on{Imm}(S^1, \R^2)$ and $k \geq 0$. Then
\begin{align}
\label{eq:Dk_logcp_var}
D_{c,h} \left(D_s^k \log |c'|\right) &= D_s^k \langle D_s h, v \rangle 
- \sum_{j=0}^{k-1} \binom{k}{j+1} \left( D_s^{k-j} \log |c'| \right) D_s^j \langle D_s h, v\rangle \\
\label{eq:Dk_al_var}
D_{c,h} \left(D_s^k \al\right) &= D_s^k \langle D_s h, n \rangle 
- \sum_{j=0}^{k-1} \binom{k}{j+1} \left( D_s^{k-j} \al \right) D_s^j \langle D_s h, v\rangle\,.
\end{align}
\end{lemma}

\begin{proof} Recall Lem.\ \ref{lem:Ds_rmap}:
if $F: \on{Imm}(S^1, \R^2) \to C^\infty(S^1, \R^d)$ is smooth then
\[
D_{c,h} \left(D_s\o F\right) = D_s\left(D_{c,h} F \right) - \langle D_s h, v \rangle D_s F(c)\,.
\]
For $k=0$, by Sect. \ref{variational-formulae} we have 
\begin{align*}
D_{c,h} \left(\log |c'| \right) &= \langle D_s h, v \rangle\,, 
\quad
D_{c,h} \al = \langle D_s h, n \rangle\,,
\quad
D_{c,h}D_s = -\langle D_sh,v \rangle D_s\,,
\\
D_{c,h}(D_s^k) &= - \sum_{j=0}^{k-1} D_s^j \o \langle D_sh,v \rangle \o D_s^{k-j}\,.
\end{align*}
Thus we get
\[
D_{c,h} \left(D_s^k \log |c'|\right) = D_s^k \langle D_s h, v \rangle - \sum_{j=0}^{k-1} D_s^j \Big( \langle D_s h, v\rangle \big( D_s^{k-j} \log |c'| \big) \Big)\,.
\]
Next we use the identity \cite[(26.3.7)]{NIST2010},
\[
\sum_{j=i}^{k-1} \binom{j}{i} = \binom{k}{i+1} \,,
\]
and the product rule for differentiation to obtain
\begin{align*}
D_{c,h} \left(D_s^k \log |c'|\right) &= D_s^k \langle D_s h, v \rangle - \sum_{j=0}^{k-1} \sum_{i=0}^j \binom{j}{i} \left( D_s^{k-j+j-i} \log |c'| \right) D_s^i \langle D_s h, v\rangle \\
&= D_s^k \langle D_s h, v \rangle - \sum_{i=0}^{k-1} \sum_{j=i}^{k-1} \binom{j}{i} \left( D_s^{k-i} \log |c'| \right) D_s^i \langle D_s h, v\rangle \\
&= D_s^k \langle D_s h, v \rangle - \sum_{i=0}^{k-1} \binom{k}{i+1} \left( D_s^{k-i} \log |c'| \right) D_s^i \langle D_s h, v\rangle\,,
\end{align*}
which completes the first part of the proof. 
Along the same lines we also get the variation of $D_s^k \al$. 
\end{proof}

\begin{theorem}
\label{thm:ka_bound}
Assume that the metric $G$ on $\on{Imm}(S^1,\R^2)$ satisfies
\begin{equation}
\label{eq:G_stronger_Dk2}
\int_{S^1} |h|^2 + |D_s^{n} h|^2 \ud s \leq A\, G_c(h,h)\,.
\end{equation}
for some $n\geq 2$ and some $A>0$.
For each $c_0 \in \on{Imm}(S^1,\R^2)$ and $N>0$ there exists a constant $C = C(c_0,N)$ such that for all $c_1, c_2 \in \on{Imm}(S^1,\R^2)$ with $\on{dist}^G(c_0,c_i) < N$ and all $0\leq k \leq n-2$ we have
\begin{gather*}
\left\| (D_{c_1}^k\ka_1) \sqrt{|c_1'|} - (D_{c_2}^k\ka_2) \sqrt{|c_2'|} \right\|_{L^2(d\th)} \leq C \on{dist}^G(c_1, c_2) \\
\left\| (D_{c_1}^{k+1}\log |c_1'|) \sqrt{|c_1'|} - (D_{c_2}^{k+1}\log |c_2'|) \sqrt{|c_2'|} \right\|_{L^2(d\th)} \leq C \on{dist}^G(c_1, c_2)\,.
\end{gather*}
In particular the functions 
\begin{align*}
(D_s^k \ka) \sqrt{|c'|} &: (\on{Imm}(S^1,\R^2), \on{dist}^G) \to L^2(S^1,\R) \\
(D_s^{k+1} \log|c'|) \sqrt{|c'|} &: (\on{Imm}(S^1,\R^2), \on{dist}^G) \to L^2(S^1,\R)
\end{align*}
are continuous and Lipschitz continuous on every metric ball.
\end{theorem}

\begin{proof}
We have $\on{dist}^G(c_1,c_2) < 2N$ by the triangle inequality. 
Let $c(t,\th)$ be a path between $c_1$ and $c_2$ with $\on{Len}^G(c) \leq 3N$. Then
\begin{align*}
\on{dist}^G(c_0, c(t)) &\leq \on{dist}^G(c_0, c_1) + \on{dist}^G(c_1, c(t)) \\
&\leq N + \on{Len}^G(c|_{[0,t]}) \\
&\leq N + 3N \leq 4N\,;
\end{align*}
thus any path of this kind remains within a ball of radius $4N$ around $c_0$.

We will prove the theorem for each $n$ by induction over $k$. The proof of the continuity of 
$(D_s^k \ka)\sqrt{|c'|}$ does not depend on the continuity of $(D_s^{k+1}\log|c'|) \sqrt{|c'|}$. 
Thus, even if we prove both statements in parallel, 
we will assume that we have established the 
continuity and local Lipschitz continuity of $ (D_s^k \ka) \sqrt{|c'|}$
when estimating 
$\|\p_t \big((D_s^{k+1} \log |c'|) \sqrt{|c'|}\big)\|_{L^2(d\th)}$ below; in particular we will need 
that 
\begin{equation}\label{eq:kabounded}
\big\| D_s^k \ka\big\|_{L^2(ds)}\quad\text{ remains bounded along the path.}
\end{equation}    

The proof consists of two steps. First we show that the following estimates hold along $c(t,\th)$:
\begin{align}
\label{eq:pt_ka}
\left \|\p_t \big((D_s^k \ka) \sqrt{|c'|}\big)\right\|_{L^2(d\th)} 
&\lesssim_{c_0, N} \left(1 + \| D_s^k \ka \|_{L^2(ds)}\right) \sqrt{G_c(c_t, c_t)} \\
\label{eq:pt_logcp}
\left \|\p_t \big((D_s^{k+1} \log |c'|) \sqrt{|c'|}\big)\right\|_{L^2(d\th)} 
&\lesssim_{c_0, N} \left(1 + \| D_s^{k+1} \log |c'|\|_{L^2(ds)}\right) \sqrt{G_c(c_t, c_t)}\,.
\end{align}
Then we apply Gronwall's inequality to prove the theorem.

{\bf Step 1.} For $k=0$ we have
\begin{align*}
\p_t \big( \ka \sqrt{|c'|}\big) &=
\langle D_s^2 c_t, n \rangle \sqrt{|c'|} - \tfrac 32 \ka \langle D_s c_t, v \rangle \sqrt{|c'|} \\
\p_t \big( (D_s \log |c'|) \sqrt{|c'|}\big) &=
\langle D_s^2 c_t, v \rangle \sqrt{|c'|} + \ka \langle D_s c_t, n \rangle \sqrt{|c'|}\, -
\\&\qquad\qquad
- \tfrac 12 (D_s \log|c'|)\langle D_s c_t, v \rangle \sqrt{|c'|}\,,
\end{align*}
and therefore
\begin{align*}
\left \|\p_t \big(\ka \sqrt{|c'|}\big)\right\|_{L^2(d\th)} &\leq
\| D_s^2 c_t \|_{L^2(ds)} + \tfrac 32 \| \ka \|_{L^2(ds)} \| D_s c_t \|_{L^\infty} \\
\left \|\p_t \big((D_s \log |c'|) \sqrt{|c'|}\big)\right\|_{L^2(d\th)} &\leq
\| D_s^2 c_t \|_{L^2(ds)} + \| \ka \|_{L^2(ds)} \| D_s c_t \|_{L^\infty} + {}\\
&\qquad {}+\tfrac12 \| D_s \log |c'| \|_{L^2(ds)} \| D_s c_t \|_{L^\infty}\,.
\end{align*}
Note that the length $\ell_c$ is bounded along $c(t,\th)$ by Cor.\ \ref{cor:ell_lip}. Using the 
Poincar\'e inequalities from Lem.\ \ref{lem:poincare} and assumption \eqref{eq:G_stronger_Dk2} we obtain
\begin{align*}
\left\|\p_t \big(\ka \sqrt{|c'|}\big)\right\|_{L^2(d\th)} &\lesssim_{c_0, N}
\left( 1 +  \| \ka \|_{L^2(ds)} \right) \sqrt{G_c(c_t, c_t)} \\
\left\|\p_t \big((D_s\log |c'|) \sqrt{|c'|}\big)\right\|_{L^2(d\th)} &\lesssim_{c_0, N}
\left( 1 + \| D_s \log |c'| \|_{L^2(ds)}\right) \sqrt{G_c(c_t, c_t)}\,.
\end{align*}
For the second estimate we used the boundedness of $\|\ka \|_{L^2(ds)}$ from \eqref{eq:kabounded}. 
This concludes the proof of step 1 for $k=0$.

Now consider $k>0$ and assume that the theorem has been shown for $k-1$. Along $c(t,\th)$ the following objects are bounded
\begin{itemize}
\item
$\ell_c$ by Cor.\ \ref{cor:ell_lip}, allowing us to use Poincar\'e inequalities,
\item
$\| D_s^{k-1} \ka \|_{L^2(ds)}$ and $\| D_s^k \log |c'| \|_{L^2(ds)}$ by induction, and
\item
$\| D_s^j \ka \|_{L^\infty}$ and $\| D_s^{j+1} \log |c'| \|_{L^\infty}$ for $0 \leq j \leq k-2$ via Poincar\'e inequalities.
\end{itemize}

We also have the following bounds, which are valid for both $v$ and $n$:
\begin{itemize}
\item
$\| D_s^j \langle D_s c_t, v \rangle \|_{L^2(ds)} \lesssim_{c_0, N} \sqrt{G_c(c_t,c_t)}$ for $0\leq j \leq k\,.$

This is clear for $j\leq k-1$, since the highest derivative of $\ka$ that appears due to the Frenet 
equations is $D_s^{k-2} \ka$ and thus all terms involving $\ka$ can be bounded by the 
$L^\infty$-norm. For $j=k$ we have  
\begin{align*}
D_s^k \langle D_s c_t, v \rangle = \langle D_sc_t, D_s^k v \rangle + \sum_{j=1}^k \binom{k}{j} \langle D_s^{j+1} c_t, D_s^{k-j} v \rangle
\end{align*}
and 
\[ 
D_s^k v = (D_s^{k-1} \ka) n + \text{lower order derivatives in $\ka$}\,.
\]
Thus
\[
\| D_s^k v \|_{L^2(ds)} \leq \| D_s^{k-1}\ka\|_{L^2(ds)} + \dots \lesssim_{c_0, N} 1\,.
\]
Hence we get
\begin{align*}
\| D_s^k \langle D_s c_t, v \rangle \|_{L^2(ds)} &\leq
\| D_s c_t \|_{L^\infty} \| D_s^k v \|_{L^2(ds)}  + \sum_{j=1}^k \binom{k}{j}
\| D^{j+1}_s c_t \|_{L^2(ds)} \| D_s^{k-j} v \|_{L^\infty} \\
&\lesssim_{c_0, N} \sqrt{G_c(c_t, c_t)}\,.
\end{align*}

\item
$\| D_s^{k+1} \langle D_s c_t, v \rangle \|_{L^2(ds)} \lesssim_{c_0, N} (1 + \| D_s^k \ka \|_{L^2(ds)} )\sqrt{G_c(c_t,c_t)}\,.$

We obtain this bound from
\begin{align*}
D_s^{k+1} \langle D_s c_t, v \rangle &=
\langle D_s^{k+2} c_t, v\rangle + \langle D_sc_t, D_s^{k+1} v \rangle +
\sum_{j=1}^k \binom{k+1}j \langle D_s^{k+2-j} c_t, D_s^{j} v\rangle\,.
\end{align*}
Taking the $L^2(ds)$-norm we get
\begin{multline*}
\| D_s^{k+1} \langle D_s c_t, v \rangle \|_{L^2(ds)} \leq
\| D_s^{k+2} c_t \|_{L^2(ds)} + \| D_s c_t \|_{L^\infty} \| D_s^{k+1} v \|_{L^2(ds)} +
\\
+
\sum_{j=1}^k \binom{k+1}j \| D_s^{k+2-j} c_t \|_{L^2(ds)} \| D_s^{j} v \|_{L^\infty} 
\\
\lesssim_{c_0,N}\! \sqrt{G_c(c_t, c_t)} + 
\left(1 + \| D_s^k \ka \|_{L^2(ds)}\right) \sqrt{G_c(c_t, c_t)} +
\sqrt{G_c(c_t, c_t)}\,,
\end{multline*}
thus showing the claim.
\end{itemize}

Equation \eqref{eq:Dk_al_var} from Lem.\ \ref{lem:ka_log_var}, rewritten for $\ka$, is
\[
D_{c,h} \left(D_s^k \ka \right) = D_s^{k+1} \langle D_s h, n \rangle 
- \sum_{j=0}^{k} \binom{k+1}{j+1} \left( D_s^{k-j} \ka \right) D_s^j \langle D_s h, v\rangle\,.
\]
Thus we get
\begin{multline*}
\p_t \left((D_s^k \ka) \sqrt{|c'|}\right) =
(D_s^{k+1} \langle D_s c_t, n \rangle) \sqrt{|c'|} - (k+\tfrac 12) (D_s^k \ka) \,\langle D_s c_t, v \rangle \sqrt{|c'|} -{} \\
{}- \binom{k+1}{2} \left(D_s^{k-1} \ka\right)\left(D_s \langle D_s c_t, v \rangle\right) \sqrt{|c'|}
-\sum_{j=2}^k \binom{k+1}{j+1} \left(D_s^{k-j} \ka\right)\, D_s^j \langle D_s c_t, v \rangle \sqrt{|c'|}\,,
\end{multline*}
and hence, by taking norms,
\begin{align*}
\left \|\p_t ((D_s^k \ka) \sqrt{|c'|})\right\|_{L^2(d\th)} &\leq
\left\| D_s^{k+1} \langle D_s c_t, n \rangle \right\|_{L^2(ds)} 
+ (k+\tfrac 12) \left\| D_s^k \ka \right\|_{L^2(ds)} 
\Big\| \langle D_s c_t, v \rangle \Big\|_{L^\infty} \\
&\qquad{}+ \binom{k+1}{2} \left\| D_s^{k-1} \ka\right\|_{L^2(ds)} 
\Big\| D_s \langle D_s c_t, v \rangle \Big\|_{L^\infty} \\
&\qquad{}+\sum_{j=2}^k \binom{k+1}{j+1} \left\| D_s^{k-j} \ka\right\|_{L^\infty} 
\Big\| D_s^j \langle D_s c_t, v \rangle\Big\|_{L^2(ds)}\\
&\lesssim_{c_0, N} \left(1 + \| D_s^k \ka \|_{L^2(ds)}\right) \sqrt{G_c(c_t, c_t)}\,.
\end{align*}
For $(D_s^{k+1} \log |c'|) \sqrt{|c'|}$ we proceed similarly. The time derivative is
\begin{multline*}
\p_t ((D_s^{k+1} \log |c'|) \sqrt{|c'|}) =
D_s^{k+1} \langle D_s c_t, v \rangle \sqrt{|c'|} -
\\
- (k+\tfrac 12) (D_s^{k+1} \log |c'|)\, \langle D_s c_t, v \rangle \sqrt{|c'|} 
\\
- \binom{k+1}{2} \left(D_s^{k} \log |c'|\right)\, D_s \langle D_s c_t, v \rangle \sqrt{|c'|} 
\\
-\sum_{j=2}^k \binom{k+1}{j+1} \left(D_s^{k+1-j} \log |c'|\right) D_s^j \langle D_s c_t, v \rangle \sqrt{|c'|}\,,
\end{multline*}
which can be estimated by
\begin{align*}
\Big\| \p_t \big((D_s^{k+1} &\log |c'|) \sqrt{|c'|}\big)\Big\|_{L^2(d\th)} \\
&{}\leq
\left\| D_s^{k+1} \langle D_s c_t, v \rangle \right\|_{L^2(ds)} 
+ (k+\tfrac 12) \left\| D_s^{k+1} \log |c'|\right\|_{L^2(ds)} 
\Big\| D_s c_t \Big\|_{L^\infty} \\
&\qquad {}+ \binom{k+1}{2} \left\| D_s^{k} \log |c'|\right\|_{L^2(ds)} 
\Big\| D_s \langle D_s c_t, v \rangle \Big\|_{L^\infty}\\
&\qquad {}+ \sum_{j=2}^k \binom{k+1}{j+1} 
\left\| D_s^{k+1-j} \log |c'| \right\|_{L^\infty}
\Big \| D_s^j \langle D_s c_t, v \rangle \Big\|_{L^2(ds)} \\
&\lesssim_{c_0, N} \left(1 + \| D_s^{k+1} \log |c'| \|_{L^2(ds)} \right) 
\sqrt{G_c(c_t,c_t)}\,.
\end{align*}

{\bf Step 2.} The proof of this step depends only on the estimates \eqref{eq:pt_ka} and \eqref{eq:pt_logcp}. 
We have a path $c(t,\th)$ between $c_1$ and $c_2$. 
We write again $D_{c_1}$ and $D_{c(t)}$ for $D_{s_{c_1}}$ and $D_{s_{c(t)}}$, respectively. 
Define the functions
\begin{align}
\label{eq:def_A_ka}
A(t) &= \left\| (D_{c_1}^k\ka_1) \sqrt{|c_1'|} - (D_{c(t)}^k\ka(t)) \sqrt{|c(t)'|} \right\|_{L^2(d\th)}\\
\label{eq:def_B_logcp}
B(t) &= \left\| (D_{c_1}^{k+1}\log |c_1'|) \sqrt{|c_1'|} - (D_{c(t)}^{k+1}\log |c(t)'|) \sqrt{|c(t)'|} \right\|_{L^2(d\th)}\,.
\end{align}
From
\[
(D_{c}^k\ka) \sqrt{|c'|}(t,\th) - (D_{c_1}^k\ka_1) \sqrt{|c_1'|}(\th) =
\int_0^t \p_t (D_s^k \ka) \sqrt{|c'|}) (\ta, \th) \ud \ta
\]
we get, by taking norms,
\begin{align*}
A(t) &\leq \int_0^t \left \|\p_t (D_s^k \ka \sqrt{|c'|})\right\|_{L^2(d\th)} \ud \ta \\
& \lesssim_{c_0, N} \int_0^t \left(1 + \| D_s^k \ka \|_{L^2(ds)}\right) \sqrt{G_c(c_t, c_t)} \ud \ta \\
& \lesssim_{c_0, N} \int_0^t \left(1 + \| D_s^k \ka_1 \|_{L^2(ds)} + A(\ta) \right) \sqrt{G_c(c_t, c_t)} \ud \ta\,.
\end{align*}
Now we use Gronwall's inequality, Cor.\ \ref{cor:gronwall_applied}, to obtain
\[
A(t) \lesssim_{c_0, N} \left(1 + \| D_s^k \ka_1 \|_{L^2(ds)}\right) \int_0^t \sqrt{G_c(c_t, c_t)} \ud \ta\,.
\]
Taking the infimum over all paths and evaluating at $t=1$ then yields almost the desired inequality
\begin{equation}
\label{eq:almost_ineq}
\left\| (D_{c_1}^k\ka_1) \sqrt{|c_1'|} - (D_{c_2}^k\ka_2) \sqrt{|c_2'|} \right\|_{L^2(d\th)}
\lesssim_{c_0, N} \left(1 + \| D_s^k \ka_1 \|_{L^2(ds)}\right) \on{dist}^G(c_1, c_2)\,.
\end{equation}
To bound $\| D_s^k \ka_1 \|_{L^2(ds)}$, which appears on the right hand side, we apply \eqref{eq:almost_ineq} with $c_2 = c_0$.
\begin{align*}
\| D_s^k \ka_1 \|_{L^2(ds)} &\leq 
\left\| D_{c_1}^k(\ka_1) \sqrt{|c_1'|} - D_{c_0}^k(\ka_0) \sqrt{|c_0'|} \right\|_{L^2(d\th)}
+ \| D_s^k \ka_0 \|_{L^2(ds)} \\
&\lesssim_{c_0, N} \left(1 + \| D_s^k \ka_0 \|_{L^2(ds)}\right) \on{dist}^G(c_0, c_1) + 
\| D_s^k \ka_0 \|_{L^2(ds)} 
\lesssim_{c_0, N} 1\,.
\end{align*}
This concludes the proof for $(D_s^k \ka) \sqrt{|c'|}$. 
For $(D_s^{k+1} \log |c'|) \sqrt{|c'|}$ proceed in the same way with $B(t)$ in place of $A(t)$ 
using the estimate \eqref{eq:pt_logcp}.
\end{proof}

\begin{remark}
Theorem \ref{thm:ka_bound} makes no statement about the continuity or local Lipschitz continuity of 
the function $\log |c'|\, \sqrt{|c'|}$, when $G$ is a Sobolev metric of order 1. In fact it appears 
that one needs a metric of order $n\geq 2$. In that case one can use the variational formula  
\[
D_{c,h} \left( \log |c'|\, \sqrt{|c'|}\right) = 
\left( 1 + \tfrac 12 \log |c'| \right) \langle D_s h, v \rangle \sqrt{|c'|}\,,
\]
and the same method of proof -- with $n \geq 2$ one can estimate $\langle D_s h, v \rangle$ using the $L^\infty$-norm -- to show that,
\[
\left(\log |c'|\right) \sqrt{|c'|} : (\on{Imm}(S^1,\R^2), \on{dist}^G) \to L^2(S^1,\R^2)
\]
is continuous and Lipschitz continuous on every metric ball.
\end{remark}

\begin{remark}
\label{rem:Linfty_continuity}
In a similar way we can also obtain continuity in $L^\infty$ instead of $L^2$. Assume the metric 
satisfies \eqref{eq:G_stronger_Dk2} with $n\geq 3$. Then for all $1 \leq k \leq n-2$ the functions 
\begin{align*}
D_s^{k-1} \ka &: (\on{Imm}(S^1,\R^2), \on{dist}^G) \to L^\infty(S^1,\R) \\
D_s^{k} \log|c'| &: (\on{Imm}(S^1,\R^2), \on{dist}^G) \to L^\infty(S^1,\R)
\end{align*}
are continuous and Lipschitz continuous on every metric ball. To prove this we follow the proof of Thm.\ \ref{thm:ka_bound} and replace the estimates \eqref{eq:pt_ka}, \eqref{eq:pt_logcp} with
\begin{align*}
\left \|\p_t \big(D_s^{k-1} \ka \big)\right\|_{L^\infty} &\lesssim_{c_0, N} \left(1 + \| D_s^{k-1} \ka \|_{L^\infty}\right) \sqrt{G_c(c_t, c_t)} \\
\left \|\p_t \big(D_s^{k} \log |c'| \big)\right\|_{L^\infty} &\lesssim_{c_0, N} \left(1 + \| D_s^{k} \log |c'|\|_{L^\infty}\right) \sqrt{G_c(c_t, c_t)}\,,
\end{align*}
which can be established in the same way.

We also have $L^\infty$-continuity of $\log |c'|$, when $n=2$. 
Since we will use it in the proof of geodesic completeness, we shall provide an explicit proof in Lem.\ \ref{lem:lenpt_inv_bound}.
\end{remark}

\begin{lemma}
\label{lem:lenpt_inv_bound}
Let the metric $G$ on $\on{Imm}(S^1,\R^2)$ satisfy
\[
\int_{S^1} |h|^2 + |D_s^n h|^2 \ud s \leq A\, G_c(h, h)\,,
\]
for some $n\geq 2$ and some $A > 0$. Given $c_0 \in \on{Imm}(S^1,\R^2)$ and $N > 0$, there exists a constant $C = C(c_0, N)$ such that for all $c_1, c_2 \in \on{Imm}(S^1,\R^2)$ with $\on{dist}^G(c_0, c_i) < N$ we have
\[
\left\| \log |c_1'| - \log |c_2'| \right\|_{L^\infty} \leq C \on{dist}^G(c_1,c_2)\,.
\]
In particular the function
\[
\log |c'| : (\on{Imm}(S^1,\R^2), \on{dist}^G) \to L^\infty(S^1,\R)
\]
is continuous and Lipschitz continuous on every metric ball.
\end{lemma}

\begin{proof}
Fix $\th \in S^1$ and $c_1 \in \on{Imm}(S^1,\R^2)$ satisfying $\on{dist}^G(c_0, c_1) < N$ and let $c(t,\th)$ be a path between $c_0$ and $c_1$ with $\on{Len}^G(c) \leq 2N$. Then
\[
\p_t \left(\log |c'(\th)|\right) =
\langle D_s c_t(\th), v(\th) \rangle\,.
\]
After integrating and taking norms we get
\begin{align*}
\left| \log |c_1'(\th)| - \log |c_0'(\th)| \right| 
&\leq \int_0^1 |D_s c_t(t, \th)| \ud t\,.
\end{align*}
Using Poincar\'e inequalities and Cor.\ \ref{cor:ell_lip} we can estimate
\begin{multline}
\label{eq:Ds_ct_explicit_bound}
|D_s c_t(\th)| \leq \frac{\sqrt{\ell_c}}2 \| D_s^2 c_t \|_{L^2(ds)} \leq
\\
\leq \frac{\sqrt{\ell_c}}2 \sqrt{ \| c_t \|_{L^2(ds)}^2 + \| D_s^n c_t \|_{L^2(ds)}^2} \leq
\tfrac 12 \sqrt{\ell_c A} \sqrt{G_c(c_t, c_t)}\,.
\end{multline}
Thus by taking the infimum over all paths between $c_0$ and $c_1$ we get
\[
\left\| \log |c_1'| - \log |c_0'| \right\|_{L^\infty} \lesssim_{c_0, N} 
\on{dist}^G(c_0, c_1) \,. \qedhere
\]
\end{proof}

\begin{remark*}
An explicit value for the constant is given by
\[
C(c_0, N) = \tfrac 12 \sqrt{A} \left(\sqrt{\ell_{c_0}} + \tfrac 12 \sqrt{A}N\right)\,.
\]
This can be found by combining the estimates \eqref{eq:Ds_ct_explicit_bound} and \eqref{eq:ell_lip_const}.
\end{remark*}

This corollary gives us upper and lower bounds on $|c'(\th)|$ in terms of the geodesic distance. 
Therefore, a geodesic $c(t,\th)$ for a Sobolev metric with order at 
least 2 cannot leave $\on{Imm}(S^1,\R^2)$ by having $c'(t,\th)=0$ for some $(t,\th)$.

\begin{corollary}
\label{cor:lenpt_inv_bound}
Under the assumptions of Lem.\ \ref{lem:lenpt_inv_bound}, given $c_0 \in \on{Imm}(S^1,\R^2)$ and $N>0$, there exists a constant $C=C(c_0,N)$, such that
\[
\| c' \|_{L^\infty} \leq C\,\qquad\text{  and }\qquad \left\| \frac1{|c'|} \right\|_{L^\infty} \leq C\,
\]
hold for all $c \in \on{Imm}(S^1,\R^2)$ with $\on{dist}^G(c_0,c)<N$.
\end{corollary}

\begin{proof}
By Lem.\ \ref{lem:lenpt_inv_bound} we have
\[
\| \log |c'(\th)| \leq
\| \log |c_0'(\th)| + \left\| \log |c'| - \log |c_0'| \right\|_{L^\infty}
 \lesssim_{c_0, N}  1\,.
\]
Now apply $\on{exp}$ and take the supremum over $\th$ to obtain $\| c'\|_{L^\infty} \lesssim_{c_0,N} 1$. Similarly by starting from
\[
-\| \log |c'(\th)| \leq
-\| \log |c_0'(\th)| + \left\| \log |c'| - \log |c_0'| \right\|_{L^\infty}
 \lesssim_{c_0, N}  1\,.
\]
we obtain the bound $\left\| |c'|\i \right\| \lesssim_{c_0,N}$.
\end{proof}

\begin{remark*}
Using the explicit constant for Lem.\ \ref{lem:lenpt_inv_bound}, we can obtain the following more explicit inequalities,
\begin{align*}
|c'(\th)| &\leq |c_0'(\th)| \exp\left( \tfrac 12 \sqrt{A}N \left( \sqrt{\ell_{c_0}} + \tfrac 12 \sqrt{A}N\right)\right) \\
|c'(\th)|\i &\leq |c_0'(\th)|\i \exp\left( \tfrac 12 \sqrt{A}N \left( \sqrt{\ell_{c_0}} + \tfrac 12 \sqrt{A}N\right)\right)\,.
\end{align*}
for Cor.\ \ref{cor:lenpt_inv_bound}.
\end{remark*}

\begin{remark}
To simplify the exposition, the results in this section were formulated on the space 
$\on{Imm}(S^1,\R^2)$ of smooth immersions. 
If $G$ is a Sobolev metric of order $n$ with $n\geq 2$, 
we can replace $\on{Imm}(S^1,\R^2)$ by $\on{Imm}^n(S^1,\R^2)$ in all statements of this section  with the same proofs.
\end{remark}

\section{Geodesic Completeness for Sobolev Metrics}

On the space $H^n(S^1,\R^d)$ we have two norms: the $H^n(d\th)$-norm as well as the $H^n(ds)$-norm, which depends on the choice of a curve $c \in \on{Imm}(S^1,\R^2)$. Although the norms are equivalent, the constant in the inequality
\[
C\i \| h \|_{H^k(d\th)} \leq \| h \|_{H^k(ds)} \leq C \| h \|_{H^k(d\th)}\,,
\]
depends in general on the curve and its derivatives. The next lemma shows, that if $c$ remains in a metric ball with respect to the geodesic distance, then the constant depends only on the center and the radius of the ball.

\begin{lemma}
\label{lem:Hk_local_equivalence}
Let the metric $G$ on $\on{Imm}(S^1,\R^2)$ satisfy
\[
\int_{S^1} |h|^2 + |D_s^n h|^2 \ud s \leq A\, G_c(h,h)
\]
for some $n \geq 2$ and some $A>0$. Given $c_0 \in \on{Imm}(S^1,\R^2)$ and $N > 0$, there exists a constant $C = C(c_0, N)$ such that for $0 \leq k \leq n$,
\[
C\i \| h \|_{H^k(d\th)} \leq \| h \|_{H^k(ds)} \leq C \| h \|_{H^k(d\th)}\,,
\]
holds for all $c \in \on{Imm}(S^1,\R^2)$ with $\on{dist}^G(c_0,c) < N$ and all $h \in H^k(S^1,\R^d)$.
\end{lemma}

\begin{proof}
By definition,
\begin{align*}
\| u \|_{H^k(d\th)}^2 &= \| h \|^2_{L^2(d\th)} + \| \p_\th^k h \|^2_{L^2(d\th)} \\
\| u \|_{H^k(ds)}^2 &= \| h \|^2_{L^2(ds)} + \| D_s^k h \|^2_{L^2(ds)}\,.
\end{align*}
The estimates
\[
\left( \min_{\th \in S^1} |c'(\th)| \right) \| h \|^2_{L^2(d\th)}
\leq \| h \|^2_{L^2(ds)} \leq
\| c'\|_{L^\infty} \| h \|^2_{L^2(d\th)}\,,
\]
together with Cor.\ \ref{cor:lenpt_inv_bound} take care of the $L^2$-terms. Thus it remains to compare the derivatives $\| \p_\th^k h \|^2_{L^2(d\th)}$ and $\| D_s^k h \|^2_{L^2(ds)}$. From the identities
\begin{align*}
h' &= |c'| D_s h \\
h'' &= |c'|^2 D_s^2 h + (\p_\th|c'|) D_sh \\
h''' &= |c'|^3 D_s^3 h + 3\,|c'|(\p_\th|c'|) D_s^2 h + (\p_\th^2 |c'|) D_sh \\
h''''\! &= |c'|^4 D_s^4 h + 6\, |c'|^2 (\p_\th |c'|) D_s^3 h + 
\left( 3\left(\p_\th |c'|\right)^2 \!+ 4\, |c'|(\p_\th^2 |c'|)\right)D_s^2 h + (\p_\th^3 |c'|) D_s h\,,
\end{align*}
we generalize to
\begin{equation}
\label{eq:hk_to_Dshk}
\p_\th^k h = \sum_{j=1}^k \sum_{\al \in A_j} c_{j,\al} \prod_{i=0}^{k-1} \left(\p_\th^i |c'|\right)^{\al_i} D_s^j h\,,
\end{equation}
where $c_{j,\al}$ are some constants and $\al=(\al_0,\ldots,\al_{k-1})$ are multi-indices that are summed over the index sets
\[
A_j = \left\{ \al\,:\, \sum_{i=0}^{k-1} i\al_i = k-j,\, \sum_{i=0}^{k-1} \al_i = j \right\}\,.
\]
Equation \eqref{eq:hk_to_Dshk} is related to Fa\`a di Bruno's formula \cite{Bruno1855} and can be proven by induction.

The length $\ell_c$ is bounded on the metric ball by Cor.\ \ref{cor:ell_lip}. 
Then Lem.\ \ref{thm:ka_bound} together with Poincar\'e inequalities shows that
\begin{itemize}
\item
$\| D_s^{n-1} \log|c'| \|_{L^2(ds)}$ and
\item
$\| D_s^{k} \log|c'| \|_{L^\infty}$ for $1 \leq k \leq n-2$
\end{itemize}
are bounded as well. Repeated application of the chain rule for differentiation yields
\begin{equation*}
D_s^k|c'| = D_s^k\left( \exp \log |c'|\right) = |c'|\, D_s^k \log |c'| + \text{lower $D_s$-derivatives of $\log |c'|$}\,.
\end{equation*}
Thus also $\| D_s^{n-1} |c'| \|_{L^2(ds)}$ and $\| D_s^{k} |c'| \|_{L^\infty}$ for $1 \leq k \leq n-2$ are bounded on metric balls. Next we apply formula \eqref{eq:hk_to_Dshk} to $h = |c'|$ obtaining
\begin{equation}
\label{eq:cp_dth}
\p_\th^k|c'| = |c'|^k D_s^k |c'| + \text{lower $D_s$-derivatives of $|c'|$}\,.
\end{equation}
Together with Lem.\ \ref{lem:lenpt_inv_bound} this implies that
\begin{itemize}
\item
$\| \p_\th^{n-1} |c'| \|_{L^2(d\th)}$ and
\item
$\| \p_\th^k |c'| \|_{L^\infty}$ for $0 \leq k \leq n-2$
\end{itemize}
are bounded on metric balls.

We proceed by induction over $k$. The case $k=0$ has been dealt with at the beginning of the proof. Assume $k \leq n-1$ and the equivalence of the norms has been shown for $k-1$. Then the highest derivative of $|c'|$ is $\p_\th^{k-1} |c'|$ and so in \eqref{eq:hk_to_Dshk} we can estimate every term involving $|c'|$ using the $L^\infty$-norm. Thus using Poincaré inequalities and the equivalence of $L^2(d\th)$ and $L^2(ds)$-norms we get
\[
\| \p_\th^k h \|^2_{L^2(d\th)} \lesssim_{c_0, N} \| D_s^k h \|^2_{L^2(ds)}\,.
\]
For the other inequality write
\[
D_s^k h = |c'|^{-k}\, \p_\th^k h - |c'|^{-k} \sum_{j=1}^{k-1} \sum_{\al \in A_j} c_{j,\al} \prod_{i=0}^{k-1} \left(\p_\th^i |c'|\right)^{\al_i} D_s^j h\,,
\]
and use the induction assumption $\| D_s^j h \|^2_{L^2(ds)} \lesssim_{c_0, N} \| \p_\th^j h \|^2_{L^2(d\th)}$ for $0 \leq j < k$.

The only remaining case is $k=n$. There we have to be a bit more careful, since then $\p_\th^{n-1} 
|c'|$ appears in \eqref{eq:hk_to_Dshk}, which cannot be bound using the $L^\infty$-norm. However 
$\p_\th^{n-1} |c'|$ appears only in the summand $\left( \p_\th^{n-1} |c'| \right) D_s h$, i.e. if 
$\al_{n-1} \neq 0$, then $\al_{n-1} = 1$, $\al_i = 0$ for $i \neq n-1$ and $\al \in A_1$. This term 
we can estimate via    
\[
\left\| \left( \p_\th^{n-1} |c'| \right) D_s h \right\|_{L^2(d\th)} \leq
\left\| \p_\th^{n-1} |c'| \right\|_{L^2(d\th)} \| D_s h\|_{L^\infty}\,,
\]
and then depending on which direction we want to estimate, we can use either of
\begin{align*}
\| D_s h \|_{L^\infty} &\leq 2\i \sqrt{\ell_c}\, \left\| D_s^2 h \right\|_{L^2(ds)} \\
\| D_s h \|_{L^\infty} &\leq \left\| |c'|\i \right\|_{L^\infty} \left\| \p_\th h \right\|_{L^\infty} \leq C \left\| |c'|\i \right\|_{L^\infty}
\left\| \p_\th^2 h \right\|_{L^2(d\th)}\,.
\end{align*}
From here we proceed as for $k< n$.
\end{proof}

We saw in Lem.\ \ref{lem:sob_mult} that multiplication is a bounded bilinear map on the spaces $H^k(S^1,\R^d)$ with the $H^k(d\th)$-norm. Since the $H^k(d\th)$-norm and the $H^k(ds)$-norm are equivalent, this holds also for the $H^k(ds)$-norm. A consequence of Lem.\ \ref{lem:Hk_local_equivalence} is that the constant in the inequality
\[
\left\| \langle f, g \rangle \right\|_{H^k(ds)} \leq C \| f \|_{H^k(ds)} \| g \|_{H^k(ds)}\,,
\]
again depends only on the center and radius of the geodesic ball.

\begin{corollary}
\label{cor:Hk_multiplication}
Under the assumptions of Lem.\ \ref{lem:Hk_local_equivalence} there exists a constant $C=C(c_0, N)$ such that for $c \in \on{Imm}(S^1,\R^2)$ with $\on{dist}^G(c_0, c) < N$ and $1 \leq k \leq n$,
\[
\left\| \langle f, g \rangle \right\|_{H^k(ds)} \leq C \| f \|_{H^k(ds)} \| g \|_{H^k(ds)}\,,
\]
holds for all $f, g \in H^k(S^1,\R^d)$. 
\end{corollary}

\begin{proof}
We use Lem.\ \ref{lem:Hk_local_equivalence} and the boundedness of multiplication on $H^k(d\th)$,
\begin{align*}
\left\| \langle f, g \rangle \right\|_{H^k(ds)} 
&\lesssim_{c_0,N} \left\| \langle f, g \rangle \right\|_{H^k(d\th)} \\
&\lesssim_{c_0,N} \| f \|_{H^k(d\th)} \| g \|_{H^k(d\th)} 
\lesssim_{c_0, N} \| f \|_{H^k(ds)} \| g \|_{H^k(ds)}\,. \qedhere
\end{align*}
\end{proof}

This last lemma shows that the identity
\[
\on{Id}: \left((\on{Imm}^n(S^1,\R^2), \on{dist}^G\right) \to \left(H^n(S^1,\R^2), H^n(d\th)\right)
\]
maps bounded sets to bounded sets and that the same holds for the function
\[
\left((\on{Imm}^n(S^1,\R^2), \on{dist}^G\right) \to \R\,,\quad
c \mapsto \| c\|_{H^n(ds)}\,,
\]
when $G$ is stronger than a Sobolev metric of order $n$.

\begin{lemma}
\label{lem:c_bounded}
Let the metric $G$ on $\on{Imm}(S^1,\R^2)$ satisfy
\[
\int_{S^1} |h|^2 + |D_s^n h|^2 \ud s \leq A\, G_c(h,h)
\]
for some $n \geq 2$ and some $A>0$. Given $c_0 \in \on{Imm}(S^1,\R^2)$ and $N > 0$, there exists a constant $C = C(c_0, N)$, such that
\[
\| c \|_{H^n(d\th)} \leq C\,,\qquad
\| c \|_{H^n(ds)} \leq C\,,
\]
hold for all $c \in \on{Imm}(S^1,\R^2)$ with $\on{dist}^G(c_0, c) < N$.
\end{lemma}

\begin{proof}
It is only necessary to prove the boundedness in one of the norms, since Lem.\ \ref{lem:Hk_local_equivalence} will imply the other one. 
We have
\[
\| c \|^2_{H^n(ds)} = \| c \|^2_{L^2(ds)} + \| D_s^n c \|^2_{L^2(ds)}
= \| c \|^2_{L^2(ds)} + \| D_s^{n-2} \ka \|^2_{L^2(ds)}\,.
\]
The boundedness of $\| D_s^{n-2} \ka \|^2_{L^2(ds)}$ on metric balls has been shown in Thm.\ \ref{thm:ka_bound}. For $\| c \|_{L^2(ds)}$ we choose a path $c(t)$ from $c_0$ to $c = c(1)$ with $\on{Len}^G(c(t)) < 2N$. Then
\begin{align*}
\| c \|_{L^2(ds)} &\lesssim_{c_0,N} \| c \|_{L^2(d\th)} 
\leq \| c -c_0 \|_{L^2(d\th)} + \| c_0 \|_{L^2(d\th)} \\
&\lesssim_{c_0, N} \left\| \int_0^1 \p_t c(t) \ud t \right\|_{L^2(d\th)} 
\leq \int_0^1 \left\| \p_t c(t) \right\|_{L^2(d\th)} \ud t \\
&\lesssim_{c_0, N} \int_0^1 \left\| \p_t c(t) \right\|_{L^2(ds)} \ud t 
\leq \on{Len}^G(c(t)) \lesssim_{c_0, N} 1. \qedhere
\end{align*}
\end{proof}

\begin{remark}
The proof of Lem.\ \ref{lem:Hk_local_equivalence} shows that under the assumptions of Lem.\ \ref{lem:c_bounded} we can choose $C=C(c_0,N)$ such that the additional inequality
\[
\left\| |c'| \right\|_{H^{n-1}(d\th)} \leq C\,,
\]
holds as well.
\end{remark}

Now we are ready to prove the main theorem.

\begin{theorem}
\label{thm:long_time}
Let $n\geq 2$ and let $G$ be a Sobolev metric with constant coefficients $a_i\ge 0$ of order $n$ 
and $a_0,a_n>0$. 
Given $(c_0, u_0) \in T\on{Imm}^n(S^1,\R^2)$ the solution of the geodesic equation for the metric $G$ with initial values $(c_0, u_0)$ exists for all time.
\end{theorem}

\begin{corollary}
Let the metric $G$ be as in Thm.\ \ref{thm:long_time}. Then the Riemannian manifolds $(\on{Imm}^n(S^1,\R^2),G)$ and $\on{Imm}(S^1,\R^2), G)$ are geodesically complete.
\end{corollary}

\begin{proof}
The geodesic completeness of $\on{Imm}(S^1,\R^2)$ follows from Thm.\ \ref{thm:geod_ex}, since given smooth initial conditions the intervals of existence are uniform in the Sobolev order.
\end{proof}

\begin{proof}[Proof of Theorem \ref{thm:long_time}]
The geodesic equation is equivalent to the following ODE on $(T\on{Imm}^n)' \cong \on{Imm}^n \x H^{-n}$,
\begin{align*}
c_{t} &= \bar{L}_c\i p \\
p_t &= \tfrac 12 H_c\left(\bar{L}_c\i p, \bar{L}_c\i p\right)\,,
\end{align*}
with $p(t) = \bar{L}_{c(t)} u(t)$. Fix initial conditions $(c(0), p(0))$. In order to show that the geodesic with these initial conditions exists for all time, we need to show that on any finite interval $[0, T)$, on which the geodesic $(c(t), p(t))$ exists, we have that
\begin{enumerate}[(A)]
\item
\label{completeness_cond1}
the closure of ${c([0,T))}$ in $H^n(S^1,\R^2)$ is contained in   $\on{Imm}^n(S^1,\R^2)$ and,
\item
\label{completeness_cond2}
$\left\| \bar{L}_c\i p \right\|_{H^n(d\th)}$, 
$\tfrac 12 \left\| H_c(\bar{L}_c\i p, \bar{L}_c\i p) \right\|_{H^{-n}(d\th)}$ are bounded on $[0,T)$.
\end{enumerate}
Then we can apply \cite[Thm.\ 10.5.5]{Dieudonne1969} to conclude that $[0,T)$ is not the maximal interval of existence. Since this holds for every $T$, the geodesic must exist on $[0,\infty)$.

Assume now that $T>0$ is fixed. We will pass freely between the momentum and the velocity via $u(t) = \bar{L}\i_{c(t)} p(t)$. Since $c(t)$ is a geodesic, we have
\[
\on{dist}^G(c_0, c(t)) \leq \sqrt{G_{c(0)}(u(0), u(0))}\, T
\quad\text{and}\quad
G_{c(t)}(u(t), u(t)) = G_{c(0)}(u(0), u(0))\,.
\]
In particular the geodesic remains in a metric ball around $c_0$. It follows from Cor.\ \ref{cor:lenpt_inv_bound} that there exists a $C>0$ with $|c'(t,\th)|\geq C$ for $(t,\th) \in [0,T) \x S^1$. Since the set $\{ c : |c'(\th)| \geq C \}$ is $H^2$-closed -- and hence also $H^n$-closed -- in $\on{Imm}^n(S^1,\R^2)$, we can conclude that condition (\ref{completeness_cond1}) is satisfied.

The first part of condition (\ref{completeness_cond2}) follows easily from
\begin{multline*}
\left\| \bar{L}_c\i p \right\|^2_{H^n(d\th)} = \| u \|^2_{H^n(d\th)}
\lesssim_{c_0,T} \| u \|^2_{H^n(ds)} \leq{} \\
\leq \max(a_0\i, a_n\i) G_c(u,u)
= \max(a_0\i, a_n\i) G_{c(0)}(u(0),u(0))\,,
\end{multline*}
using Lem.\ \ref{lem:Hk_local_equivalence} and that the velocity is constant along a geodesic.

It remains to show that $\| H_c(u, u) \|_{H^{-n}(d\th)}$ remains bounded along $c(t)$. To estimate 
this norm, pick $m \in H^n(d\th)$ and consider the pairing 
\begin{multline*}
\langle H_c(u,u), m \rangle_{H^{-n}\x H^n} = D_{c,m} G_c(u,u)  
= \int_{S^1} \sum_{k=0}^n a_k \langle D_s^k u, D_s^k u \rangle
\langle D_s m, v \rangle \ud s -
\\
- 2 \sum_{k=1}^n \sum_{j=1}^k a_k \langle D_s^k u, 
D_s^{k-j}\left( \langle D_s m, v \rangle D_s^j u \right) \rangle \ud s\,.
\end{multline*}
Using Poincar\'e inequalities, Lem.\ \ref{lem:Hk_local_equivalence}, and that $\ell_c$ is bounded 
along $c(t)$, we can estimate the first term, 
\begin{align*}
\left| \int_{S^1} \sum_{k=0}^n a_k \langle D_s^k u, D_s^k u \rangle
\langle D_s m, v \rangle \ud s \right| 
&\leq \| D_s m \|_{L^\infty}\, G_c(u,u) \\
&\lesssim_{c_0,T} \| m \|_{H^n(ds)}
\lesssim_{c_0,T} \| m \|_{H^n(d\th)}\,.
\end{align*}
For the second term we additionally need Cor.\ \ref{cor:Hk_multiplication}. For each $1\leq k \leq n$ and $1 \leq j \leq k$ we have,
\begin{align*}
\bigg| \int_{S^1} \Big\langle D_s^k u, 
D_s^{k-j}\big( \langle D_s m, v \rangle &D_s^j u \big) \Big\rangle \ud s \bigg|
\leq \left\| D_s^k u \right\|_{L^2(ds)} \left\|  D_s^{k-j}\left( 
\langle D_s m, v \rangle D_s^j u \right) \right\|_{L^2(ds)} \\
&\leq \| u \|_{H^k(ds)} \left\| \langle D_s m, v \rangle 
D_s^j u \right\|_{H^{k-j}(ds)} \\
&\lesssim_{c_0,T} \| u \|_{H^k(ds)} \| D_s m \|_{H^{k-j}(ds)} \| v \|_{H^{k-j}(ds)}
\| D_s^j u \|_{H^{k-j}(ds)} \\
&\lesssim_{c_0,T} \| u \|^2_{H^n(ds)} \| c \|_{H^{n}(ds)} \| m \|_{H^{n}(ds)}\,.
\end{align*}
We know that $\| u \|^2_{H^n(ds)}$ is bounded along $c(t)$ and using Lem.\ \ref{lem:c_bounded} we see that $\| c \|_{H^{n}(ds)}$ is bounded as well. Hence we obtain
\[
\left| \langle H_c(u,u), m \rangle_{H^{-n}\x H^n} \right|
\lesssim_{c_0,T} \| m \|_{H^n(d\th)}\,,
\]
which implies
$$
\| H_c(u,u) \|_{H^{-n}(d\th)} \lesssim_{c_0,T} 1\,,
$$
i.e., $\| H_c(u,u) \|_{H^{-n}(d\th)}$ is bounded along the geodesic. 
\end{proof}

\begin{remark}
\label{rem:incomplete}
If $G$ is a Sobolev-type metric of order $n\geq 2$ with $a_0=0$, $a_1=0$, then $G$ is a Riemannian metric on the space $\on{Imm}(S^1,\R^2)/\on{Tra}$ of plane curves modulo translations. We will show that for these metrics it is possible to blow up circles to infinity along geodesics in finite time, making them geodesically incomplete. Thus a non-vanishing zero or first order term is necessary for geodesic completeness.

The 1-dimensional submanifold consisting of all concentric circles, which are parametrized by 
constant speed, is a geodesic with respect to the metric,
because Sobolev-type metrics are invariant under the motion group. 
 Let $c(t,\th) = r(t)\, (\cos \th, 
\sin\th)$. Then $c_t(t,\th) = r_t(t)\, (\cos \th, \sin \th)$ and $|c'(t,\th)| = r(t)$. Thus  
\[
G_c(c_t,c_t) = 2\pi \sum_{j=2}^n a_j r(t)^{1-2j} r_t(t)^2\,,
\]
and the length of the curve is
\[
\on{Len}^G(c) = \int_0^1 \sqrt{2\pi \sum_{j=2}^n a_j r(t)^{1-2j} r_t(t)^2} \ud t
= \sqrt{2\pi} \int_{r(0)}^{r(1)} \sqrt{\sum_{j=2}^n a_j \si^{1-2j}} \ud \si\,.
\]
Since the integral converges for $r(1) \to \infty$, it follows that the path consisting of growing circles can reach infinity with finite length.
\end{remark}


\begin{thebibliography}{10}

\bibitem{Adams2003}
R.~A. Adams.
\newblock {\em Sobolev Spaces}.
\newblock Academic Press, 2nd edition edition, 2003.

\bibitem{Bauer2013_preprint}
M.~Bauer, M.~Bruveris, and P.~W. Michor.
\newblock Overview of the geometries of shape spaces and diffeomorphism groups.
\newblock {\em J. Math. Imaging Vis.}, 2014.
\newblock DOI: 10.1007/s10851-013-0490-z

\bibitem{Bauer2013d_preprint}
M.~Bauer, M.~Bruveris, and P.~W. Michor.
\newblock {$R$}-transforms for {S}obolev {$H^2$}-metrics on spaces of plane
  curves.
\newblock {\em Geom. Imaging Comput.}, 1(1):1--56, 2014.

\bibitem{Bauer2011b}
M.~Bauer, P.~Harms, and P.~W. Michor.
\newblock Sobolev metrics on shape space of surfaces.
\newblock {\em J. Geom. Mech.}, 3(4):389--438, 2011.

\bibitem{Bauer2012d}
M.~Bauer, P.~Harms, and P.~W. Michor.
\newblock Sobolev metrics on shape space, II: Weighted {S}obolev metrics and
  almost local metrics.
\newblock {\em J. Geom. Mech.}, 4(4):365 -- 383, 2012.

\bibitem{Bruveris2014}
M.~Bruveris and F.-X. Vialard.
\newblock On completeness of groups of diffeomorphisms.
\newblock 2014. Eprint: {\tt arXiv:1403.2089}

\bibitem{Charpiat2007}
G.~Charpiat, P.~Maurel, J.-P. Pons, R.~Keriven, and O.~Faugeras.
\newblock Generalized gradients: Priors on minimization flows.
\newblock {\em Int. J. Comput. Vision}, 73(3):325--344, 2007.

\bibitem{Dieudonne1969}
J.~Dieudonn{\'e}.
\newblock {\em Foundations of modern analysis}.
\newblock Academic Press, New York, 1969.
\newblock Enlarged and corrected printing, Pure and Applied Mathematics, Vol.
  10-I.

\bibitem{Ebin1970}
D.~G. Ebin and J.~Marsden.
\newblock Groups of diffeomorphisms and the motion of an incompressible fluid.
\newblock {\em Ann. of Math. (2)}, 92:102--163, 1970.

\bibitem{Bruno1855}
F.~Fa{\`a}~di Bruno.
\newblock Note sur une nouvelle formule du calcul diff{\'e}ntielle.
\newblock {\em Quart. J. Math.}, 1:359--360, 1855.

\bibitem{Inci2013}
H.~Inci, T.~Kappeler, and P.~Topalov.
\newblock {\em On the Regularity of the Composition of Diffeomorphisms}, volume
  226 of {\em Memoirs of the American Mathematical Society}.
\newblock American Mathematical Society, 2013.

\bibitem{Jones1964}
G.~S. Jones.
\newblock Fundamental inequalities for discrete and discontinuous functional
  equations.
\newblock {\em J. Soc. Indust. Appl. Math.}, 12:43--57, 1964.

\bibitem{KM97}
A.~Kriegl and P.~W. Michor.
\newblock{\em The convenient setting of global analysis},
volume~53 of {\em Mathematical Surveys and Monographs}.
\newblock American Mathematical Society,   Providence, RI, 1997. 

\bibitem{Kurtek2012}
S.~Kurtek, A.~Srivastava, E.~Klassen, and Z.~Ding.
\newblock Statistical modeling of curves using shapes and related features.
\newblock {\em J. Am. Stat. Assoc.}, 107(499):1152--1165, 2012.

\bibitem{Mennucci2008}
A.~Mennucci, A.~Yezzi, and G.~Sundaramoorthi.
\newblock Properties of {S}obolev-type metrics in the space of curves.
\newblock {\em Interfaces Free Bound.}, 10(4):423--445, 2008.

\bibitem{Micheli2013}
M.~Micheli, P.~W. Michor, and D.~Mumford.
\newblock Sobolev metrics on diffeomorphism groups and the derived geometry of
  spaces of submanifolds.
\newblock {\em Izvestiya: Mathematics}, 77(3):541--570, 2013.

\bibitem{Michor2006c}
P.~W. Michor and D.~Mumford.
\newblock Riemannian geometries on spaces of plane curves.
\newblock {\em J. Eur. Math. Soc. (JEMS) 8 (2006), 1-48}, 2006.

\bibitem{Michor2007}
P.~W. Michor and D.~Mumford.
\newblock An overview of the {R}iemannian metrics on spaces of curves using the
  {H}amiltonian approach.
\newblock {\em Appl. Comput. Harmon. Anal.}, 23(1):74--113, 2007.

\bibitem{Michor2013}
D.~Mumford and P.~W. Michor.
\newblock On {E}uler's equation and '{EPDiff}'.
\newblock {\em J. Geom. Mech.}, 5(3):319--344, 2013.

\bibitem{Nirenberg1959}
L.~Nirenberg.
\newblock On elliptic partial differential equations.
\newblock {\em Ann. Scuola Norm. Sup. Pisa (3)}, 13:115--162, 1959.

\bibitem{NIST2010}
F.~W.~J. Olver, D.~W. Lozier, R.~F. Boisvert, and C.~W. Clark, editors.
\newblock {\em N{IST} Handbook of Mathematical Functions}.
\newblock U.S. Department of Commerce National Institute of Standards and
  Technology, Washington, DC, 2010.

\bibitem{Pachpatte1998}
B.~G. Pachpatte.
\newblock {\em Inequalities for differential and integral equations}, volume
  197 of {\em Mathematics in Science and Engineering}.
\newblock Academic Press Inc., San Diego, CA, 1998.

\bibitem{Sundaramoorthi2011}
G.~Sundaramoorthi, A.~Mennucci, S.~Soatto, and A.~Yezzi.
\newblock A new geometric metric in the space of curves, and applications to
  tracking deforming objects by prediction and filtering.
\newblock {\em SIAM J. Imaging Sci.}, 4(1):109--145, 2011.

\bibitem{Mennucci2007}
G.~Sundaramoorthi, A.~Yezzi, and A.~C. Mennucci.
\newblock Sobolev active contours.
\newblock {\em Int. J. Comput. Vision}, 73(3):345--366, 2007.

\bibitem{TrouveYounes05}
A.~Trouv{\'e} and L.~Younes.
\newblock Local geometry of deformable templates.
\newblock {\em SIAM J. Math. Anal.}, 37(1):17--59 (electronic), 2005.

\end{thebibliography}
\end{document}